\documentclass[12pt]{amsart}

\usepackage{epsfig}
\usepackage{mathtools}
\usepackage[subnum]{cases}
\usepackage{enumerate}
\usepackage{amsmath,amssymb}
\usepackage{color}
\usepackage[bookmarks=true,
bookmarksnumbered=true, breaklinks=true,
pdfstartview=FitH, hyperfigures=false,
plainpages=false, naturalnames=true,
colorlinks=true,pagebackref=true,
pdfpagelabels]{hyperref}

\newtheorem{theorem}{Theorem}[section]
\newtheorem{definition}[theorem]{Definition}
\newtheorem{lemma}[theorem]{Lemma}
\newtheorem{corollary}[theorem]{Corollary}
\newtheorem{remark}[theorem]{Remark}
\newtheorem{proposition}[theorem]{Proposition}

\newtheorem{conjecture}[theorem]{Conjecture}

\newcommand{\R}{\mathbb{R}} 

\newcommand{\e}{\varepsilon} 

\newcommand{\N}{\mathbb{N}}

\begin{document}
	\title{On $L_p$-Brunn-Minkowski type and $L_p$-isoperimetric type inequalities for  measures}
	\author[Michael Roysdon and Sudan Xing]
	{Michael Roysdon and Sudan Xing}

	\address{Department of Mathematical Sciences, Kent State University,
	 USA} \email{mroysdon@kent.edu }

	\address{Department of Mathematical and Statistical Sciences, University of Alberta, Canada} \email{sxing@ualberta.ca}

	\subjclass[2010]{Primary: 52A39, 52A40, 46N10;  Secondary: 28A75, 26D15} \keywords{$L_p$-Brunn-Minkowski inequality, $L_p$-Borell-Brascamp-Lieb inequality,  $L_p$-Pr\'ekopa-Leindler inequality,
	$L_p$-Isoperimetric inequality}

\date{}

\maketitle

\begin{abstract}
In 2011 Lutwak, Yang and Zhang extended the definition of the $L_p$-Minkowski convex combination ($p \geq 1$) introduced by Firey in the 1960s from convex bodies containing the origin in their interiors to all measurable subsets in $\R^n$, and as a consequence, extended the $L_p$-Brunn-Minkowski inequality ($L_p$-BMI) to the setting of all measurable sets.  In this paper, we present a functional extension of their $L_p$-Minkowski convex combination---the $L_{p,s}$--supremal convolution and prove the $L_p$-Borell-Brascamp-Lieb type ($L_p$-BBL) inequalities. Based on the $L_p$-BBL type inequalities for functions, we extend the $L_p$-BMI for measurable sets to the class of Borel measures on $\R^n$ having $\left(\frac{1}{s}\right)$-concave densities, with $s \geq 0$; that is, we show that, for any pair of Borel sets $A,B \subset \R^n$, any $t \in [0,1]$ and  $p\geq 1$, one has 
\[
\mu((1-t) \cdot_p A +_p t \cdot_p B)^{\frac{p}{n+s}} \geq (1-t) \mu(A)^{\frac{p}{n+s}} + t \mu(B)^{\frac{p}{n+s}},
\]
where $\mu$ is a measure on $\R^n$ having a $\left(\frac{1}{s}\right)$-concave density for  $0 \leq s < \infty$. 

Additionally, with the new defined $L_{p,s}$--supremal convolution for functions, we prove $L_p$-BMI for product measures with quasi-concave densities and  for log-concave densities, $L_p$-Pr\'ekopa-Leindler type inequality ($L_p$-PLI)  for product measures with quasi-concave densities, $L_p$-Minkowski's first inequality ($L_p$-MFI) and $L_p$ isoperimetric inequalities ($L_p$-ISMI) for general measures, etc.
Finally a functional counterpart of the Gardner-Zvavitch conjecture is presented for the $p$-generalization. 
\end{abstract}

\section{Introduction}
Let $\R^n$ denote the $n$-dimensional Euclidean space equipped with the usual norm $\|\cdot\|$ and its usual inner product structure $\langle \cdot, \cdot \rangle$. For a function $f\colon \R^n \to \R_+$ we denote its support by $\text{supp}(f)$. For measurable subsets $A, B \subset \R^n$ and $\alpha, \beta >0$ the Minkowski combination of the sets $A$ and $B$ with respect to the constants $\alpha$ and $\beta$ is defined by 
\[
\alpha A + \beta B = \{\alpha x+ \beta y \colon x \in A, y \in B\}. 
\]
A result involving the Minkowski combination is the famed Brunn-Minkowski inequality, which asserts that, for any $t \in[0,1]$ and for any measurable sets $A,B\subset \R^n$ such that the sets $(1-t)A + tB$ is also measurable, one has 
\begin{equation}\label{e:BM}
|(1-t) A+tB|_{n}^{1/n} \geq (1-t) |A|_n^{1/n} + t |B|_n^{1/n}.
\end{equation}
If, in addition, $A$ and $B$ are taken to be convex sets, then equality in \eqref{e:BM} holds if and only if $A$ and $B$ are homothetic. Here $|\cdot|_n$ denotes the $n$-dimensional Lebesgue measure (or volume) on $\R^n$.  The Brunn-Minkowski inequality has lead to a rich theory which links probability, analysis, and convex geometry (cf. \cite{AGM, Kold, Schneider:CB2} and the survey \cite{Gar}). Many generalizations of the Brunn-Minkowski inequality have been studied in the cases of measures (cf. \cite{BORELL, CLM, GZ, KL, LMNZ, arno1}) and in the functional setting (cf. \cite{BrascampLieb, L, P}). 

Let $\alpha \in [-\infty,\infty]$. We say that a Borel measure $\mu$ on $\R^n$ is $\alpha$-concave if, for any Borel sets $A,B \subset \R^n$ and any $t \in [0,1]$, one has 
\[
\mu((1-t)A+tB) \geq M_{\alpha}^t(\mu(A),\mu(B)),
\]
where, for $a,b \geq 0$,
\[
M_{\alpha}^t(a,b) =\begin{cases}
((1-t)a^{\alpha}+tb^{\alpha})^{1/\alpha} &\text{if } \alpha\neq 0, \pm \infty\\
a^{1-t}b^t &\text{if } \alpha =0\\
\max\{a,b\} &\text{if } \alpha = \infty\\
\min\{a,b\} &\text{if } \alpha =-\infty,
\end{cases}
\]
if $ab >0$, and $M_{\alpha}^t(a,b) = 0$ if $ab=0$. In the case when $\alpha =0$ the measure $\mu$ is often referred to as a $\log$-concave measure. Similarly, we say that a function $f \colon \R^n \to \R_+$ is $\alpha$-concave if, for any $x,y \in \R^n$ and every $t \in [0,1]$, one has 
\[
f((1-t)x + t y) \geq M_{\alpha}^t(f(x),f(y)).
\]
Moreover, when $\alpha = - \infty$, such functions are referred to as quasi-concave functions.  In fact, a bounded function $f \colon \R^n \to \R_+$ is quasi-concave if and only if its super-level sets 
\[
C_{r}(f) =\{x\in\R^n \colon f(x) \geq r\|f\|_{\infty}\}
\]
are convex for every $0 \leq r \leq 1$.

A seminal extension of the Brunn-Minkowski inequality for convex bodies is the Borell-Brascamp-Lieb inequality for functions, studied by Borell in \cite{BORELL} and by Brascamp and Lieb in \cite{BrascampLieb}. 

\begin{theorem}[Borell-Brascamp-Lieb inequality (BBL)]\label{t:BBL} Let $\alpha$ be such that $\alpha \geq -\frac{1}{n}$ and $t \in [0,1]$. Suppose that $f,g,h\colon \R^n \to \R_+$ is a triple of measurable functions that satisfy the condition 
\begin{equation}\label{e:BBLassumption}
h((1-t)x+ty) \geq M_{\alpha}^t(f(x),g(y))
\end{equation}
for every $x,y \in \R^n$. Then the following inequality holds true:
\begin{equation}\label{e:BBLconclusion}
\int_{\R^n}h(x) dx \geq M_{\gamma}^t\left(\int_{\R^n}f(x)dx, \int_{\R^n}g(x)dx\right), \quad \gamma=\frac{\alpha}{1+n\alpha}.
\end{equation}
\end{theorem}

The case $\alpha =0$ is referred to as the Pr\'ekopa-Leindler inequality and was explored by Pr\'ekopa in \cite{P} and Leindler in \cite{L}.

Theorem~\ref{t:BBL} implies, in particular, that measures having $\alpha=\left(\frac{1}{s}\right)$-concave density must be $\left( \frac{1}{n+s}\right)$-concave for $s\leq -n$ and $s\geq0$. The minimal function satisfying the condition \eqref{e:BBLassumption} is the so-called $(t,s)$-supremal convolution, that is, the measurable function $h_{t,s} \colon \R^n \to \R_+$ defined by 
\begin{equation}\label{e:supconvolution}
\begin{split}
h_{t,s}(z) &= \sup_{z = (1-t)x + t y} M_{\frac{1}{s}}^t(f(x),g(y))\\
&=\sup_{z = (1-t)x + t y} \left[(1-t)f(x)^{\frac{1}{s}} + t g(y)^{\frac{1}{s}}\right]^s,
\end{split}
\end{equation}
where $f,g \colon \R^n \to \R_+$ are arbitrary measurable functions. Alternatively,  the $(t,s)$-supremal convolution can be expressed via an operation of addition and scalar multiplication. Given two measurable functions $,f,g \colon \R^n \to \R_+$, the sum $f \oplus_s g$  is defined by \cite{BCF, Klartag} as
\[
[f \oplus_s g](x) = \sup_{z = x+y}[f(x)^{1/s}+g(y)^{1/s}]^s,
\]
where the supremum is taken over all way to write $z=x+y$, with $x \in \text{supp}(f)$ and $y \in \text{supp}(g)$. Moreover, given a positive constant $\alpha$, the scalar of this sum $\times_s$ satisfies
\[
[\alpha \times_s f](x) = \alpha^sf(x/\alpha).
\]
Therefore, we really have $h_{t,s} = (1-t) \times_s f \oplus_s t \times_s g$. 
Moreover, the function $h_{t,s}$ is $\left(\frac{1}{s}\right)$-concave whenever $f,g$ are as well. In particular, whenever $f=1_A$ and $g =1_B$ are characteristic functions of measurable sets $A,B \subset \R^n$ with $(1-t)A+tB$ also measurable, then 
\[
h_{t,s}(z) = 1_{(1-t)A+tB}(z).
\]
In this sense, the $(t,s)$-supremal convolution extends naturally the Minkowski convex combination to the functional setting.  

In \cite{Firey} Firey extended the Minkowski combination in the setting of convex bodies $\mathcal{K}_{(o)}^n$ (convex, compact subsets of $\R^n$ with the origin in their interiors), which has been named as the $L_p$-Minkowski convex combination (cf. \cite{Schneider:CB2}). Let $p \in [1,\infty)$ and $t \in [0,1]$. Then, for any convex bodies $K,L \subset \mathcal{K}_{(o)}^n$, and $\alpha, \beta >0$, the $L_p$-Minkowski combination of $K$ and $L$ with respect to $\alpha$ and $\beta$ is defined to be  the convex body $\alpha \cdot_p K +_p \beta\cdot_p L$ whose support function is given by 
\begin{equation}\label{e:Lpminkowskicomb}
h_{\alpha \cdot_p K +_p \beta\cdot_p L}(x) = (\alpha h_K(x)^p + \beta h_L(x)^p)^{1/p},
\end{equation}
where, for a convex body $K$, $h_{K}(x) = \max_{y \in K} \langle x,y \rangle$ is the support function of $K$. Note that when $p=1$, equation \eqref{e:Lpminkowskicomb} becomes the usual Minkowski convex combination. In \cite{Firey} an extension of the Brunn-Minkowski inequality, often referred to as the $L_p$-Brunn-Minkowski inequality, was established for convex bodies containing the origin in their interiors. In \cite{LYZ} Lutwak, Yang, and Zhang extended the definition \eqref{e:Lpminkowskicomb} to the collection of all measurable subsets of $\R^n$ that coincides with original definition due to Firey when the sets involved are convex bodies containing the origin in their interiors. Let $p \in [1,\infty)$ and $t \in [0,1]$. Given non-empty measurable subsets $A$ and $B$ of $\R^n$, and $\alpha, \beta >0$, the $L_p$-Minkowski combination of $A$ and $B$  with respect to $\alpha$ and $\beta$ is defined by 
\begin{equation}\label{e:LYZMinkowskicombo}
\alpha \cdot_p A +_p \beta \cdot_p B = \left\{\alpha^{\frac{1}{p}}(1-\lambda)^{\frac{p-1}{p}}x+\beta^{\frac{1}{p}}\lambda^{\frac{p-1}{p}}y \colon x \in A, y \in B, 0\leq \lambda \leq 1\right\}.
\end{equation}
As a consequence of this extension of the $L_p$-Minkowski convex combination, Lutwak, Yang, and Zhang extended the  $L_p$-Brunn-Minkowski inequality to the setting of all non-empty measurable sets.

\begin{theorem}[$L_p$-Brunn-Minkowski inequality ($L_p$-BMI)] Let $p \in[1,\infty)$ and $t \in [0,1]$. For any non-empty measurable subsets $A,B$ of $\R^n$ such that $(1-t) \cdot_p A +_p t \cdot_p B$ is also measurable, one has 
\begin{equation}\label{e:LYZBM}
|(1-t) \cdot_p A +_p t \cdot_p B|_n \geq M_{\frac{p}{n}}^t(|A|_n,|B|_n).
\end{equation}
In addition, for $p >1$, if $A,B \in\mathcal{K}_{(o)}^n$, then equality occurs in \eqref{e:LYZBM} only when $A$ and $B$ are dilates of one another.
\end{theorem}

This study, ignited by Firey, has been widely pushed forward by Lutwak in \cite{lutwak, lutwak2}. Additionally, the famous $\log$-Brunn-Minkowski conjecture of B\"or\"oczky, Lutwak,  Yang, and Zhang  \cite{Boroczky-Lutwak-Yang-Zhang} has been regarded as a result of Firey's generalization.  For more advances of the $L_p$-Brunn-Minkowski theory see also \cite{BHZ, BLYZ-1,BLYZ-2,BHT, global,CLM,KolMilsupernew,LYZ,arno,Putter,liran,christos, christos1, Stan,Stan1,YgZg1,Zhu,Zou}.

Similar to the $L_p$-Minkowski combination of convex bodies, we define the $L_{p,s}$--supremal convolution for functions, that is, $h_{p,t,s} \colon \R^n \to \R_+$ is defined by 
\begin{equation}\label{e:newsupcolvolution}
h_{p,t,s}(z)=
\sup_{0\leq \lambda \leq 1}\left( \sup_{z=(1-t)^{\frac{1}{p}}(1-\lambda)^{\frac{p-1}{p}}x + t^{\frac{1}{p}}\lambda^{\frac{p-1}{p}}y} \left[(1-t)^{\frac{1}{p}}(1-\lambda)^{\frac{p-1}{p}} f(x)^{\frac{1}{s}} + t^{\frac{1}{p}}\lambda^{\frac{p-1}{p}}g(y)^{\frac{1}{s}}\right]^s\right),   
\end{equation}
where the supremum combined together represents the supremum taken over 
\[
z \in (1-t)\cdot_p \text{supp}(f) +_p t \cdot_p \text{supp}(g)
\]
for $t\in[0,1].$

In the same manner as the classical Borell-Brascamp-Lieb for functions can be used to derive the Brunn-Minkowski inequality for convex bodies, in 
Section $2$ we present the $L_p$-Brunn-Minkowski inequalities 
for measures with $(1/s)$-concave densities based on $L_p$-BBL inequality for functions, and $L_p$-BMI for product measures with quasi-concave densities based on $L_p$-PLI for product measures with quasi-concave densities, respectively. That is,

\begin{theorem}[$L_p$-BMI for measures with $(1/s)$-concave densities]\label{t:Lpsconcavemeasures} Let $p \in[1,\infty)$, $t \in [0,1]$ and $s \in [0,\infty)$. Let $\mu$ be a measure given by $d\mu(x) = \phi(x) dx$, where $\phi \colon \R^n \to \R_+$ is a $\left(\frac{1}{s}\right)$-concave function on its support. Then, for any Borel subsets $A$ and $B$ in $\R^n$, one has 
\begin{equation}\label{e:LpBMmeasures}
\mu((1-t) \cdot_p A +_p t \cdot_p B) \geq M_{\frac{p}{n+s}}^t(\mu(A),\mu(B)). 
\end{equation}

\end{theorem}

We note that the parameter of the above theorem is  $0 \leq s < \infty$ while  the case $s =\infty$ is obvious and it follows immediately from the Borell-Brascamp-Lieb inequality \eqref{e:BBLconclusion} and the inclusion (cf. \cite{LYZ})
\[
(1-t) \cdot_p A +_p t \cdot_p B \supset (1-t) A + tB.
\]
Therefore,  
we develop the following improved $L_p$-BMI for the case $s = \infty$, but before stating the result, we require the following definition from \cite{JesusManuel}. We say that set $A \subset \R^n$ is weakly unconditional if, for every $x=(x_1,\dots,x_n) \in A$ and every $\e = (\e_1,\dots,\e_n) \in \{0,1\}^n$, one has 
\[
(\e_1x_1,\dots,\e_nx_n) \in A.
\]

\begin{theorem}[$L_p$-BMI for product measures with quasi-concave densities] \label{t:LpPLmeasures0}  Let $p \in (1,\infty)$, $t\in [0,1]$, and $\mu = \mu_1 \times \cdots \times \mu_n$ be a product measure on $\R^n$, where, for each $i = 1,\dots,n$, $\mu_i$ is a measure on $\R$ having a quasi-concave density with maximum at the origin. Then for any weakly unconditional measurable sets $A,B \subset \R^n$, such that $(1-t) \cdot_p A +_p t \cdot_p B$ is also measurable, one has 
\begin{equation}\label{e:LpPLmeasures}
\mu((1-\lambda)\cdot_p A +_p \lambda \cdot_p B)^{\frac{p}{n}} \geq (1-\lambda) \mu(A)^{\frac{p}{n}} + \lambda \mu(B)^{\frac{p}{n}}.
\end{equation}
\end{theorem}

In Sections $3$ and $4$, the detailed proofs of $L_p$-Borell-Brascamp-Lieb inequality for functions and $L_p$-Pr\'ekopa-Leindler type inequality for product measures with quasi-concave densities are presented, respectively. That is,
$L_p$-BBL inequality for functions states:\emph{
	Let $p \in [1,\infty)$, $0 \leq s < \infty$, and $t \in [0,1]$. Let $f,g,h \colon \R^n \to \R_+$ be a triple of bounded integrable functions. Suppose, in addition, that this triple satisfies the condition
	\begin{equation}\label{e:LpBBLassumption}
	h\left((1-t)^{\frac{1}{p}}(1-\lambda)^{\frac{p-1}{p}}x + t^{\frac{1}{p}}\lambda^{\frac{p-1}{p}}y\right) \geq \left[(1-t)^{\frac{1}{p}}(1-\lambda)^{\frac{p-1}{p}} f(x)^{\frac{1}{s}} + t^{\frac{1}{p}}\lambda^{\frac{p-1}{p}}g(y)^{\frac{1}{s}}\right]^s
	\end{equation}
	for every $x \in \text{supp}(f)$, $y\in \text{supp}(g)$ and every $\lambda \in [0,1]$. Then the following integral inequality holds:
	\begin{equation}\label{e:LPBBLinequality0}
	\int_{\R^n}h(x)dx \geq M_{\frac{p}{n+s}}^t\left(\int_{\R^n}f(x)dx,\int_{\R^n}g(x)dx \right);
	\end{equation}}
 and $L_p$-PLI for product measures with quasi-concave densities states as: \emph{Let $p > 1$, $t \in [0,1]$, and $\mu = \mu_1 \times \cdots \times \mu_n$ be a product measure on $\R^n$, where, for each $i = 1,\dots,n,$ $\mu_i$ is a measure on $\R$ having a quasi-concave density $\phi \colon \R \to \R_+$ with maximum at origin. Let $f,g,h \colon \R^n\to \R_+$ be a triple of measurable functions, with $f,g$ weakly unconditional and positively decreasing, that satisfy the condition
	\begin{equation}\label{e:lpprekopaleindlerassumption0}
	h((1-t)^{1/p}(1-\lambda)^{(p-1)/p}x + t^{1/p}\lambda^{(p-1)/p}y) \geq f(x)^{(1-t)^{1/p}(1-\lambda)^{(p-1)/p}}g(y)^{t^{1/p}\lambda^{(p-1)/p}}
	\end{equation}
	for every $x \in \text{supp}(f), y \in \text{supp}(g)$, and every $0 < \lambda < 1$. The the following integral inequality holds:
	\begin{equation}\label{e:lpprekopaleindlerconclusion0}
	\int_{\R^n} h d\mu \geq \sup_{0 < \lambda < 1}\left\{\left[\left(\frac{1-t}{1-\lambda}\right) \left(\frac{t}{\lambda}\right)\right]^{\frac{n}{p}}\left(\int_{\R^n} f^{\left(\frac{1-t}{1-\lambda}\right)^{\frac{1}{p}}} d\mu \right)^{1-\lambda}\left(\int_{\R^n}g^{\left(\frac{t}{\lambda}\right)^{\frac{1}{p}}} d\mu \right)^{\lambda} \right\}.
	\end{equation}}
 For $p=1$, formula (\ref{e:LPBBLinequality0}) recovers the classical BBL inequality (\ref{e:BBLconclusion}) and if $\mu$ is the Lebesgue measure, (\ref{e:lpprekopaleindlerconclusion0}) goes back to the classical Pr\'ekopa-Leindler inequality.

 In Section $5$, similar to the definition of surface area for convex bodies, the $L_p$-$\mu$-surface area of a  measure $\mu$-integrable function is defined with respect to $L_{p,s}$--supremal convolution for functions.
 Let $\mu$ be a Borel measure on $\R^n$, $p \geq 1$, and $s \in [0, \infty]$. We define the $L_p$-$\mu$-surface area of a $\mu$-integrable function $f \colon \R^n \to \R_+$ with respect to a $\mu$-integrable function $g$ by 
 \[
 S_{\mu,p,s}(f,g) := \liminf_{\e \to 0^+} \frac{\int_{\R^n} f \oplus_{p,s} (\e \times_{p,s} g) d\mu -\int_{\R^n}f d\mu }{\e}.
 \]
Detailed proof of the  $L_p$ functional counterparts of $L_p$-Minkowski's first inequality ($L_p$-MFI for measures in Theorem \ref{t:functionalMink1st}) for $ S_{\mu,p,s}(f,g)$ and $L_p$-isoperimetric type inequality ($L_p$-ISMI for measures in Theorem \ref{t:lpisoperimetricinequalitygeneral}) are presented. 

It is curious to know whether Theorem~\ref{t:LpPLmeasures0} can be proved, up to some constant for the log concave measure instead of the product measure.  One typical method is that it can also be solved through considering such a problem for functions.  Therefore, Section~$6$ concerns the establishment of  a functional counterpart of the $L_p$-Gardner-Zvavitch conjecture (cf. \cite{EM,GZ,HKL,KL});
that is, we show a special $L_p$-BMI for $s=\infty$ of the  $L_{p,s}$--supremal convolution of normalized log-concave functions; i.e.,
	let $p \geq 1, t \in [0,1]$, and $f,g \in \mathcal{L}^n$.  Then
\begin{equation}\label{e:lpfuncxtionalGZ}
\begin{split}
\left(\frac{1}{\|(1-t) \cdot_{p,\infty}f \oplus_{p,\infty} t \cdot_{p,\infty} g\|_{\infty}}\ \int_{\R^n}(1-t) \cdot_{p,\infty}f \oplus_{p,\infty} t \cdot_{p,\infty} g dx\right)^{\frac{p}{n}}\geq\\
\frac{1}{C^p}\cdot \left[(1-t) \left( \frac{1}{\|f\|_{\infty}}\int_{\R^n} f(x)dx\right)^{\frac{p}{n}} + t \left(\frac{1}{\|g\|_{\infty}}\int_{\R^n} g(x)dx\right)^{\frac{p}{n}}\right]
\end{split}
\end{equation}
where $C>1$ is an absolute constant. A natural corollary with this inequality for functions derived our results for measures with the uniform constant bound in Corollary \ref{t:LpGZmeasures} which is similar to Theorem \ref{t:LpPLmeasures0} up to a constant $1/C^p$.
In another way, the inequality turns out to be quite obvious when the measures are taken to have radial decay with respect to the scalar of Minkowski convex combination for convex bodies, i.e., a measure satisfying $\mu(tA) \geq t^n \mu(A)$ for any convex set $A$ containing the origin, with the constant $C = 2$ (see Remark~\ref{t:GZremark} below).

\section{ $L_p$-Brunn-Minkowski type inequalities  for measures}

In this section we propose that the $L_p$-Borell-Brascamp-Lieb inequality ($L_p$-BBL) for $s \in [0,+\infty)$ (Theorem \ref{t:LPBBL0}) for functions and an $L_p$-Pr\'ekopa-Leindler type inequality ($L_p$-PLI) for product measures with quasi-concave densities (Theorem \ref{t:LpPrekopaLeindler0}), each of which, respectively, leads to $L_p$-BMI for measures with $(1/s)$-concave densities in Theorem~\ref{t:Lpsconcavemeasures} and $L_p$-BMI for product measures with quasi-concave densities in Theorem~\ref{t:LpPLmeasures}.  We leave the proof of $L_p$-BBL type inequality for $s \in [0,+\infty)$  in Section \ref{t:LPBBL333} and proof of $L_p$-PLI for product measures with quasi-concave densities in Section \ref{pplleindlerienquality}, respectively. Firstly we assume the following $L_p$-BBL inequality holds true. 

\begin{theorem}[$L_p$-BBL inequality for functions] \label{t:LPBBL0} 
Let $p \in [1,\infty)$, $0 \leq s < \infty$, and $t \in [0,1]$. Let $f,g,h \colon \R^n \to \R_+$ be a triple of bounded integrable functions. Suppose, in addition, that this triple satisfies the condition
\begin{equation}\label{e:LpBBLassumption0}
h\left((1-t)^{\frac{1}{p}}(1-\lambda)^{\frac{p-1}{p}}x + t^{\frac{1}{p}}\lambda^{\frac{p-1}{p}}y\right) \geq \left[(1-t)^{\frac{1}{p}}(1-\lambda)^{\frac{p-1}{p}} f(x)^{\frac{1}{s}} + t^{\frac{1}{p}}\lambda^{\frac{p-1}{p}}g(y)^{\frac{1}{s}}\right]^s
\end{equation}
for every $x \in \text{supp}(f)$, $y\in \text{supp}(g)$ and every $\lambda \in [0,1]$. Then the following integral inequality holds:
\begin{equation}\label{e:LPBBLinequality1}
\int_{\R^n}h(x)dx \geq M_{\frac{p}{n+s}}^t\left(\int_{\R^n}f(x)dx,\int_{\R^n}g(x)dx \right).
\end{equation}
\end{theorem}

The minimal function $h$ which satisfies the condition \eqref{e:LpBBLassumption0} shall be referred to as the $L_{p,s}$--supremal convolution; that is, the function $h_{p,t,s} \colon \R^n \to \R_+$ defined by 
\begin{equation}\label{e:newsupcolvolution}
h_{p,t,s}(z)=
\sup_{0\leq \lambda \leq 1}\left( \sup_{z=(1-t)^{\frac{1}{p}}(1-\lambda)^{\frac{p-1}{p}}x + t^{\frac{1}{p}}\lambda^{\frac{p-1}{p}}y} \left[(1-t)^{\frac{1}{p}}(1-\lambda)^{\frac{p-1}{p}} f(x)^{\frac{1}{s}} + t^{\frac{1}{p}}\lambda^{\frac{p-1}{p}}g(y)^{\frac{1}{s}}\right]^s\right),   
\end{equation}
where the supremum combined together represents the supremum taken over 
\[
z \in (1-t)\cdot_p \text{supp}(f) +_p t \cdot_p \text{supp}(g)
\]
 and we set $h_{p,t,s}(z)=0$ otherwise. Note that this definition is well defined whenever $f,g \colon \R^n \to \R_+$ are measurable functions and $s \in [-\infty,\infty]$; the limits $s = 0,\pm \infty$ should be understood in the usual sense. Moreover, when $p=1$, $h_{p,t,s}$ recovers the usual $(t,s)$-supremal convolution appearing in \eqref{e:supconvolution}.

Here we define a new interpretation of the $L_p$ sum of functions---$L_{p,s}$--supremal convolution as follows.

\begin{definition}
Let $p \in [1,\infty)$, $s \in [0,\infty]$, and $\alpha,\beta >0$. Given non-negative functions $f,g \colon \R^n \to \R_+$, we define the addition $\oplus_{p,s}$ of $f$ and $g$ as follows 
\[
[f \oplus_{p,s} g](z) = \sup_{0 \leq \lambda \leq 1}\left(\sup_{z = (1-\lambda)^{(p-1)/p}x + \lambda^{(p-1)/p}y}(1-\lambda)^{(p-1)/p} f(x)^{1/s} + \lambda^{(p-1)/p} g(y)^{1/s}\right)^s,
\]
where the second supremum is taken over all ways to write $z = (1-\lambda)^{(p-1)/p}x+\lambda^{(p-1)/p}y$ with $x \in \text{supp}(f)$ and $y \in \text{supp}(g)$.  

Additionally, given any scalar $\alpha > 0$, it can be checked easily that the scalar multiplication $\times_{p,s}$ satisfies
\[
(\alpha \times_{p,s} f)(x) = \alpha^{s/p} f\left(\frac{x}{\alpha^{1/p}} \right).
\]
More generally, given $\alpha, \beta >0$, we define the $L_{p,s}$--supremal convolution of $f$ and $g$ with respect to the constants $\alpha$ and $\beta$ by $\alpha \times_{p,s} f \oplus_{p,s} \beta \times_{p,s} g$. 
\end{definition}

Therefore, when $\alpha=1-t$ and $\beta = t$ for some $t \in [0,1]$, denote the $L_{p,s}$--supremal convolution with perimeter $t$ as
\[h_{p,t,s}:= (1-t) \cdot_{p,s} f \oplus_{p,s} t \cdot_{p,s} g.
\]

\begin{proposition} Let $p >1$, $t \in [0,1]$, $s \in [-\infty,\infty]$. Then we have that 
\[
(1-t) \cdot_{p,s} f \oplus_{p,s} t \cdot_{p,s}g
\]
is $\left(\frac{1}{s}\right)$-concave whenever $f,g \colon \R^n \to \R_+$ are $\left(\frac{1}{s}\right)$-concave as well.

\end{proposition}

\begin{proof} Follows the similar idea of \cite[Proposition 2.1]{BCF}, we give the proof by the definition of $L_{p,s}$--supremal convolution with parameter $t$.  First, assume that $s \neq \pm \infty$. For each $0\leq \lambda \leq 1$, denote
\[
u_{\lambda}(x,y) := \left[(1-t)^{\frac{1}{p}}(1-\lambda)^{\frac{p-1}{p}}f(x)^{\frac{1}{s}}+t^{\frac{1}{p}}\lambda^{\frac{1}{p}}g(y)^{\frac{1}{s}} \right]^s, \quad x \in \text{supp}(f), \quad y \in \text{supp}(g),
\]
and for $z \in (1-t) \cdot_p \text{supp}(f) +_p t \cdot_p \text{supp}(g),$ let 
\begin{align*}
m(z) &:=[(1-t) \cdot_{p,s} f \oplus_{p,s} t \cdot_{p,s}g](z)\\
&=\sup_{0\leq \lambda \leq 1}\left( \sup_{z=(1-t)^{\frac{1}{p}}(1-\lambda)^{\frac{p-1}{p}}x + t^{\frac{1}{p}}\lambda^{\frac{p-1}{p}}y} \left[(1-t)^{\frac{1}{p}}(1-\lambda)^{\frac{p-1}{p}} f(x)^{\frac{1}{s}} + t^{\frac{1}{p}}\lambda^{\frac{p-1}{p}}g(y)^{\frac{1}{s}}\right]^s\right),
\end{align*}
and $m(z) =0$ otherwise. 

For each fixed $\lambda \in [0,1]$, we claim that the functions $u_{\lambda}$ is $\left(\frac{1}{s}\right)$-concave on $\text{supp}(f) \times \text{supp}(g)$. Set $(x,y) =(1-t')(x_1,y_1) + t'(x_2,y_2)$ for some $t' \in [0,1]$, $x_1,x_2 \in \text{supp}(f)$ and $y_1,y_2 \in \text{supp}(g)$. Then we have 
\begin{align*}
u_{\lambda}(x,y)&=\left\{(1-t)^{\frac{1}{p}}(1-\lambda)^{\frac{p-1}{p}}f((1-t')x_1+t'x_2)^{\frac{1}{s}}+t^{\frac{1}{p}}\lambda^{\frac{1}{p}}g((1-t')y_1+t'y_2)^{\frac{1}{s}} \right\}^s\\
&\geq \left\{(1-t)^{\frac{1}{p}}(1-\lambda)^{\frac{p-1}{p}}[(1-t')f(x_1)^{\frac{1}{s}} + t'f(x_2)^{\frac{1}{s}}]  + t^{\frac{1}{p}}\lambda^{\frac{p-1}{p}}[(1-t')g(y_1)^{\frac{1}{s}} + t'g(y_2)^{\frac{1}{s}}]\right\}^s\\
&=\Big\{(1-t')[(1-t)^{\frac{1}{p}}(1-\lambda)^{\frac{p-1}{p}}f(x_1)^{\frac{1}{s}}+t^{\frac{1}{p}}\lambda^{\frac{p-1}{p}}g(y_1)^{\frac{1}{s}}]\\
&+t'[(1-t)^{\frac{1}{p}}(1-\lambda)^{\frac{p-1}{p}}f(x_2)^{\frac{1}{s}}+t^{\frac{1}{p}}\lambda^{\frac{p-1}{p}}g(y_2)^{\frac{1}{s}}]\Big\}^s\\
&=\left[(1-t')u_{\lambda}(x_1,y_1)^{\frac{1}{s}} + t'u_{\lambda}(x_2,y_2)^{\frac{1}{s}}\right]^s,
\end{align*}
which means exactly that $u_{\lambda}$ is $\left(\frac{1}{s}\right)$-concave on its support for any fixed $\lambda \in [0,1]$, as claimed.  

Now fix a decomposition $z = (1-t')z_1+t'z_2$, with $z_1,z_2$ belonging to $(1-t) \cdot_p \text{supp}(f) +_p t \cdot_p \text{supp}(g).$ Using truncation, we  assume that all the functions involved are bounded. Then, given $\e>0$, choose $x_1,x_2 \in \text{supp}(f)$, $y_1,y_2 \in \text{supp}(g)$, and $0 \leq \lambda\leq 1$ such that 
\[
z_1 = (1-t)^{\frac{1}{p}}(1-\lambda)^{\frac{p-1}{p}}x_1 + t^{\frac{1}{p}}\lambda^{\frac{p-1}{p}}y_1, \quad z_2 = (1-t)^{\frac{1}{p}}(1-\lambda)^{\frac{p-1}{p}}x_2 + t^{\frac{1}{p}}\lambda^{\frac{p-1}{p}}y_2,
\]
and such that 
\[
m(z_1) \leq u_{\lambda}(x_1,y_1)+\e, \quad m(z_2) \leq u_{\lambda}(x_2,y_2)+\e.
\]
Since each function $u_{\lambda}$ is $\left(\frac{1}{s}\right)$-concave, by setting $x= (1-t')x_1+t'x_2$ and $y= (1-t')y_1+t'y_2$, we see that 
\[
u_{\lambda}(x,y) \geq M_{\frac{1}{s}}^{t'}(u_{\lambda}(x_1,y_1),u_{\lambda}(x_2,y_2)) \geq M_{\frac{1}{s}}^{t'}((m(z_1)-\e)_+, (m(z_2)-\e)_+),
\]
where $a_+=\max\{a,0\}.$

Finally, we note that $(1-t)^{\frac{1}{p}}(1-\lambda)^{\frac{p-1}{p}}x+t^{\frac{1}{p}}\lambda^{\frac{p-1}{p}}y = (1-t')z_1+t'z_2 = z$, which implies that $u_{\lambda}(x,y) \leq m(z),$ completing the proof for $-\infty<s<+\infty$.  The case $s = \pm \infty$ follows similarly. 
\end{proof}

\begin{remark} Let $f,g \colon \R^n \to \R_+$ be measurable functions, $p \in [1,\infty)$, and  $t \in [0,1]$.
\begin{enumerate}[(a)]
	
 \item For any $p >1$, one has $h_{p,t,s}(z) \geq h_{1,t,s}(z):=h_{t,s}(z)$.
 Moreover, one may apply Theorem~\ref{t:BBL} to obtain the following inequality \emph{(}which is weaker than inequality \eqref{e:LPBBLinequality1}\emph{)}
\begin{equation}\label{e:weakLpBBL}
\int_{\R^n}h_{p,t,s}(x) dx \geq M_{\frac{1}{n+s}}\left(\int_{\R^n}f(x)dx,\int_{\R^n}g(x)dx\right),
\end{equation}
whenever $\frac{1}{s} \geq - \frac{1}{n}$.
\item Let $f = 1_{A}$ and $g= 1_B$ be characteristic functions of Borel sets $A,B \subset \R^n$, respectively.   The $L_{p,s}$--supremal convolution for $p\geq1$ extends the $L_p$-Minkowski convex combination to the functional setting in the sense that 
\begin{eqnarray}\label{e:LpextensionofMinkowski}
h_{p,t,s}(z) \nonumber&=& \sup_{0\leq \lambda \leq 1}\left( \sup_{z=(1-t)^{\frac{1}{p}}(1-\lambda)^{\frac{p-1}{p}}x + t^{\frac{1}{p}}\lambda^{\frac{p-1}{p}}y} \left[(1-t)^{\frac{1}{p}}(1-\lambda)^{\frac{p-1}{p}} 1_A(x)^{\frac{1}{s}} + t^{\frac{1}{p}}\lambda^{\frac{p-1}{p}}1_B(y)^{\frac{1}{s}}\right]^s\right)\\
&=& \sup_{0 \leq \lambda \leq 1}\left( \sup_{z=(1-t)^{\frac{1}{p}}(1-\lambda)^{\frac{p-1}{p}}x + t^{\frac{1}{p}}\lambda^{\frac{p-1}{p}}y, \ x\in A,\  y\in B} \left[(1-t)^{\frac{1}{p}}(1-\lambda)^{\frac{p-1}{p}}+ t^{\frac{1}{p}}\lambda^{\frac{p-1}{p}}\right]^s\right)\\
\nonumber&=& 1_{(1-t) \cdot_p A +_p t \cdot_p B}(z).
\end{eqnarray}
\end{enumerate}
\end{remark}
To check the third equality, using (a), we see that 
\[
h_{t,p,s}(z) \geq h_{t,1,s}(z) = 1_{(1-t)A+tB}(z) = 1.
\]
In order to check that $h_{p,t,s}(z) \leq 1$, it is enough to check that the function $w(t,\lambda) = (1-t)^{\frac{1}{p}}(1-\lambda)^{\frac{p-1}{p}} + t^{\frac{1}{p}}\lambda^{\frac{p-1}{p}}$ is bounded from above by $1$; this can be seen by checking values of $w$ on the boundary of $[0,1] \times [0,1]$ and  its derivative or using the H\"older inequality directly as
\[
\left[(1-t)^{\frac{1}{p}}(1-\lambda)^{\frac{p-1}{p}}+ t^{\frac{1}{p}}\lambda^{\frac{p-1}{p}}\right]\leq [(1-t)+t]^{1/p}[(1-\lambda)+\lambda]^{\frac{p-1}{p}}=1.
\]

We are now prepared to prove the $L_p$-BBL inequality for measures having $(1/s)$-concave densities assuming that Theorem \ref{t:LPBBL0} holds true.
	
\begin{theorem}[$L_p$-BMI for measures with $(1/s)$-concave densities]\label{t:Lpsconcavemeasures} Let $p \in[1,\infty)$, $t \in [0,1]$ and $s \in [0,\infty)$. Let $\mu$ be a measure given by $d\mu(x) = \phi(x) dx$, where $\phi \colon \R^n \to \R_+$ is a $\left(\frac{1}{s}\right)$-concave function on its support. Then, for any Borel subsets $A$ and $B$ in $\R^n$, one has 
	\begin{equation}\label{e:LpBMmeasures}
	\mu((1-t) \cdot_p A +_p t \cdot_p B) \geq M_{\frac{p}{n+s}}^t(\mu(A),\mu(B)). 
	\end{equation}
	
\end{theorem}
\begin{proof}Consider the functions 
\[
f=1_A \phi, \quad g = 1_B \phi, \quad h= 1_{(1-t)\cdot_p A +_p t \cdot_p B}\phi
\]
for Borel sets $A$ and $B$.
We need to show that the above triple of functions satisfy condition \eqref{e:LpBBLassumption0} in order to give the proof of this theorem. Let $z \in (1-t) \cdot_p \text{supp}(f)+_p t \cdot_p \text{supp}(g)$ be arbitrary. Then $z=(1-t)^{\frac{1}{p}}(1-\lambda)^{\frac{p-1}{p}}x + t^{\frac{1}{p}}\lambda^{\frac{p-1}{p}}y$ for some $x \in \text{supp}(f)$, $y \in \text{supp}(g) $, and some $\lambda \in [0,1]$. Our goal is to check that
\begin{align*}
h(z)^{\frac{1}{s}} &\geq (1-t)^{\frac{1}{p}}(1-\lambda)^{\frac{p-1}{p}} f(x)^{\frac{1}{s}} + t^{\frac{1}{p}}\lambda^{\frac{p-1}{p}}g(y)^{\frac{1}{s}}\\
&=(1-t)^{\frac{1}{p}}(1-\lambda)^{\frac{p-1}{p}}\phi(x)^{\frac{1}{s}} + t^{\frac{1}{p}}\lambda^{\frac{p-1}{p}} \phi(y)^{\frac{1}{s}}. 
\end{align*}
Since $(1-t)^{\frac{1}{p}}(1-\lambda)^{\frac{p-1}{p}} \in [0,1]$, using the concavity of $\phi^{\frac{1}{s}}$, we see that $h$ must satisfy
\begin{align*}
&h(z)^{1/s}\\ &= h\left((1-t)^{\frac{1}{p}}(1-\lambda)^{\frac{p-1}{p}}x + t^{\frac{1}{p}}\lambda^{\frac{p-1}{p}}y \right)^{\frac{1}{s}}\\
&=\phi\left((1-t)^{\frac{1}{p}}(1-\lambda)^{\frac{p-1}{p}}x + t^{\frac{1}{p}}\lambda^{\frac{p-1}{p}}y\right)^{\frac{1}{s}}1_{(1-t)\cdot_pA+_pt\cdot_pB}\left((1-t)^{\frac{1}{p}}(1-\lambda)^{\frac{p-1}{p}}x + t^{\frac{1}{p}}\lambda^{\frac{p-1}{p}}y\right)\\
&\geq (1-t)^{\frac{1}{p}}(1-\lambda)^{\frac{p-1}{p}} \phi(x)^{\frac{1}{s}}+\left(1 - (1-t)^{\frac{1}{p}}(1-\lambda)^{\frac{p-1}{p}} \right) \phi\left( \frac{t^{\frac{1}{p}}\lambda^{\frac{p-1}{p}}}{1 - (1-t)^{\frac{1}{p}}(1-\lambda)^{\frac{p-1}{p}}} \cdot y \right)^{\frac{1}{s}}\\
&\geq (1-t)^{\frac{1}{p}}(1-\lambda)^{\frac{p-1}{p}} \phi(x)^{\frac{1}{s}} +t^{\frac{1}{p}}\lambda^{\frac{p-1}{p}} \phi(y)^{\frac{1}{s}}\quad \\
&=(1-t)^{\frac{1}{p}}(1-\lambda)^{\frac{p-1}{p}} (\phi(x)1_A(x))^{\frac{1}{s}} +t^{\frac{1}{p}}\lambda^{\frac{p-1}{p}} ((\phi(y)1_B(y)))^{\frac{1}{s}}\\
&=(1-t)^{\frac{1}{p}}(1-\lambda)^{\frac{p-1}{p}} f(x)^{\frac{1}{s}} +t^{\frac{1}{p}}\lambda^{\frac{p-1}{p}} g(y)^{\frac{1}{s}},
\end{align*}
where, we have used the fact that
\[
\frac{t^{\frac{1}{p}}\lambda^{\frac{p-1}{p}}}{1 - (1-t)^{\frac{1}{p}}(1-\lambda)^{\frac{p-1}{p}}} \leq 1 \iff (1-t)^{\frac{1}{p}}(1-\lambda)^{\frac{p-1}{p}} + t^{\frac{1}{p}}\lambda^{\frac{p-1}{p}} \leq 1,
\]
as desired. Hence, applying Theorem~\ref{t:LPBBL0} to the triple $f,g,h$, we conclude that 
\begin{align*}
\mu((1-t)\cdot_pA +_p t\cdot_p B) &= \int_{\R^n}h(x) dx\\
&\geq M_{\frac{p}{n+s}}^t \left(\int_{\R^n}f(x)dx, \int_{\R^n}g(x)dx \right)\\
&= M_{\frac{p}{n+s}}^t\left(\mu(A), \mu(B) \right),
\end{align*}
completing the proof.

\end{proof}

In order to establish $L_p$-BMI for product measures with quasi-concave densities (Theorem~\ref{t:LpPLmeasures}), we require a Pr\'ekopa-Leindler type inequality.  We begin with the following definition inspired by the work \cite{JesusManuel}. 
A function $f \colon \R^n \to \R_+$ is weakly unconditional if, for any $x=(x_1,\dots,x_n) \in \R^n$ and any $\e = (\e_1,\dots,\e_n) \in \{0,1\}^n$, one has 
\[
f(\e x) = f(\e_1 x_1, \dots, \e_n x_n) \geq f(x_1,\dots,x_n) = f(x). 
\]

\begin{theorem}[$L_p$-PLI for product measures with quasi-concave densities] \label{t:LpPrekopaLeindler0}  Let $p > 1$, $t \in [0,1]$, and $\mu = \mu_1 \times \cdots \times \mu_n$ be a product measure on $\R^n$, where, for each $i = 1,\dots,n,$ $\mu_i$ is a measure on $\R$ having a quasi-concave density $\phi \colon \R \to \R_+$ with maximum at origin. Let $f,g,h \colon \R^n\to \R_+$ be a triple of measurable functions, with $f,g$ weakly unconditional and positively decreasing, that satisfy the condition
\begin{equation}\label{e:lpprekopaleindlerassumption1}
h((1-t)^{1/p}(1-\lambda)^{(p-1)/p}x + t^{1/p}\lambda^{(p-1)/p}y) \geq f(x)^{(1-t)^{1/p}(1-\lambda)^{(p-1)/p}}g(y)^{t^{1/p}\lambda^{(p-1)/p}}
\end{equation}
for every $x \in \text{supp}(f), y \in \text{supp}(g)$, and every $0 < \lambda < 1$. The the following integral inequality holds:
\begin{equation}\label{e:lpprekopaleindlerconclusion1}
\int_{\R^n} h d\mu \geq \sup_{0 < \lambda < 1}\left\{\left[\left(\frac{1-t}{1-\lambda}\right) \left(\frac{t}{\lambda}\right)\right]^{\frac{n}{p}}\left(\int_{\R^n} f^{\left(\frac{1-t}{1-\lambda}\right)^{\frac{1}{p}}} d\mu \right)^{1-\lambda}\left(\int_{\R^n}g^{\left(\frac{t}{\lambda}\right)^{\frac{1}{p}}} d\mu \right)^{\lambda} \right\}.
\end{equation}

\end{theorem}

Take typical functions in above theorem, we obtain $L_p$-PLI for product measures with quasi-concave densities in terms of convex bodies in the following way.

\begin{corollary}\label{t:goodlemma}
Let $p \in (1,\infty)$, $t \in [0,1]$, and $\mu = \mu_1 \times \cdots \times \mu_n$ be a product measure on $\R^n$, where, for each $i = 1,\dots,n$, $\mu_i$ is a measure on $\R$ having a quasi-concave density with maximum at the origin. Then for any weakly unconditional measurable sets $A,B \subset \R^n$, such that $(1-t) \cdot_p A +_p t \cdot_p B$ is also measurable, one has 
\[
\mu((1-t) \cdot_p A +_p t \cdot_p B) \geq \sup_{0 < \lambda <1} \left\{\left[\left(\frac{1-t}{1-\lambda}\right)^{1-\lambda} \left(\frac{t}{\lambda}\right)^{\lambda}\right]^{ \frac{n}{p}}\mu(A)^{1-\lambda}\mu(B)^{\lambda}\right\}.
\]
\end{corollary}

\begin{proof}
Consider the functions $f = 1_A, g=1_B$, and $h=1_{(1-t) \cdot_p A +_p t \cdot_p B}$. It is easy to check that the functions $f$ and $g$ are weakly unconditional, and that the the triple $f,g,h$ satisfies the assumption \eqref{e:lpprekopaleindlerassumption1} of Theorem~\ref{t:LpPrekopaLeindler0}, which yields the desired conclusion. 
\end{proof}

One immediate consequence of Corollary \ref{t:goodlemma} is the following $L_p$-BMI for product measures with quasi-concave densities for convex bodies.
\begin{theorem}[$L_p$-BMI for product measures with quasi-concave densities]\label{t:LpPLmeasures}  Let $p \in (1,\infty)$, $t\in [0,1]$, and $\mu = \mu_1 \times \cdots \times \mu_n$ be a product measure on $\R^n$, where, for each $i = 1,\dots,n$, $\mu_i$ is a measure on $\R$ having a quasi-concave density with maximum at the origin. Then for any weakly unconditional measurable sets $A,B \subset \R^n$, such that $(1-t) \cdot_p A +_p t \cdot_p B$ is also measurable, one has 
	\begin{equation}\label{e:LpPLmeasures}
	\mu((1-\lambda)\cdot_p A +_p \lambda \cdot_p B)^{\frac{p}{n}} \geq (1-\lambda) \mu(A)^{\frac{p}{n}} + \lambda \mu(B)^{\frac{p}{n}}.
	\end{equation}

\end{theorem}
\begin{proof} We may assume that $t\in (0,1).$
Suppose that $\mu(A)\mu(B)>0$ and set
\[
1-\lambda = \frac{(1-t)\mu(A)^{\frac{p}{n}}}{(1-t)\mu(A)^{\frac{p}{n}}+ t \mu(B)^{\frac{p}{n}} }. 
\]
Then 
\[
\frac{1-t}{1-\lambda} = \frac{(1-t)\mu(A)^{\frac{p}{n}}+ t \mu(B)^{\frac{p}{n}}}{\mu(A)^{\frac{p}{n}}}, \quad \frac{t}{\lambda} = \frac{(1-t)\mu(A)^{\frac{p}{n}}+ t \mu(B)^{\frac{p}{n}}}{\mu(B)^{\frac{p}{n}}}.
\]
Consequently, applying Corollary~\ref{t:goodlemma}, we see that 
\[
\left[\left(\frac{1-t}{1-\lambda}\right)^{1-\lambda}\left(\frac{t}{\lambda}\right)^{\lambda} \right]^{\frac{n}{p}}\mu(A)^{1-\lambda}\mu(B)^{\lambda}= \left[(1-t) \mu(A)^{\frac{p}{n}} + t \mu(B)^{\frac{p}{n}}\right]^{\frac{n}{p}},
\]
which is the desired inequality in this case. 

Without loss of generality, suppose now that $\mu(B) = 0$. Using the fact that $A,B$ are weakly unconditional, we see that $0 \in B$, and moreover, $(1-t)\cdot_p A +_p t \cdot_p B \supset (1-t) \cdot_p A$. Let $\phi$ be a quasi-concave function with maximum at the origin which is the density of $\mu$. Then
\begin{align*}
\mu((1-t) \cdot_p A+_p t \cdot_p B)  &\geq \mu((1-t) \cdot_p A)\\
&= \int_{(1-t)^{1/p}A} \phi(y) dy\\
&=(1-t)^{\frac{n}{p}}\int_{A}\phi((1-t)^{1/p}x) dx\\
&=(1-t)^{\frac{n}{p}}\int_{A}\phi\left((1-t)^{\frac{1}{p}}x_1,\dots, (1-t)^{\frac{1}{p}} x_n \right)dx\\
&\geq (1-t)^{\frac{n}{p}}\int_A \phi(x_1,\dots,x_n)dx = (1-t)^{\frac{n}{p}}\mu(A).\\
\end{align*}
Consequently, we see that 
\[
\mu((1-t) \cdot_p A+_p t \cdot_p B)^{\frac{p}{n}} \geq (1-t)\mu(A)^{\frac{p}{n}} + t \mu(B)^{\frac{p}{n}},
\]
as desired.
\end{proof}
\section{$L_p$-Borell-Brascamp-Lieb inequality for functions }\label{t:LPBBL333}
The proof of $L_p$-BBL inequality for functions (Theorem~\ref{t:LPBBL0} or Theorem ~\ref{t:LPBBL}) is adopted from the ideas of Klartag in \cite{Klartag} (see also \cite{RS}). The main idea is to replace the functions with bodies of revolution and apply $L_p$-BMI for convex bodies \eqref{e:LYZBM}  in order to derive the $L_p$-BBL inequality. The proof is split into two main steps: the case when $s$ is a positive integer, and then forward to the case when $s$ is a positive rational number in lowest terms which results into the case for general nonnegative real number by approximation from rational number sequence.
We begin with the $L_p$-BBL inequality for functions with  $s\in \mathbb{N}$ in the following proposition. 

\begin{proposition}\label{t:integerlpbbl0} Let $p \in[1,\infty)$, $s \in \N$, and $t\in[0,1]$. Let $f,g,h \colon \R^n \to \R_+$ be a triple of bounded integrable functions. Suppose, in addition, that this triple satisfies the condition
\begin{equation}\label{e:integerLpBBLassumption0}
h\left((1-t)^{\frac{1}{p}}(1-\lambda)^{\frac{p-1}{p}}x + t^{\frac{1}{p}}\lambda^{\frac{p-1}{p}}y\right) \geq \left[(1-t)^{\frac{1}{p}}(1-\lambda)^{\frac{p-1}{p}} f(x)^{\frac{1}{s}} + t^{\frac{1}{p}}\lambda^{\frac{p-1}{p}}g(y)^{\frac{1}{s}}\right]^s
\end{equation}
for every $x \in \text{supp}(f)$, $y\in \text{supp}(g)$ and every $\lambda \in [0,1]$. Then the following integral inequality holds:
\begin{equation}\label{e:integerLPBBL}
\int_{\R^n}h(x)dx \geq M_{\frac{p}{n+s}}^t\left(\int_{\R^n}f(x)dx,\int_{\R^n}g(x)dx \right).
\end{equation}
\end{proposition}

Fix $s \in \N$. Given any bounded, integrable function $w \colon \R^n \to \R_+$, consider the $n+s$ dimensional body of revolution
\begin{equation}\label{e:integerbodyofrev}
A_s(w):= \left\{(x,y) \in \R^n \times \R^s \colon x \in \text{supp}(w), \|y\|\leq w(x)^{\frac{1}{s}}\right\}.
\end{equation}
We remark that $A_s(w)$ is a convex body if and only if $w^{\frac{1}{s}}$  is a concave function on its support. Additionally, using Fubini's theorem, we see that 
\[
|A_s(w)|_{n+s} = |B_2^s|_s \int_{\R^n} w(x)dx,
\]
where $B_2^s$ denotes the closed Euclidean unit ball in $\R^s$. 

In the following lemma, we present the set inclusion relation between the $n+s$ dimensional body of revolution and the $L_{p,s}$--supremal convolution for functions.
\begin{lemma}\label{t:setinclusion1} Let $p\geq 1$, $s \in \N$, and $t \in [0,1]$.  Suppose that $f,g,h\colon \R^n \to \R_+$ are bounded integrable functions that satisfy the condition \eqref{e:integerLpBBLassumption0} of Proposition~\ref{t:integerlpbbl0}. Then
\begin{equation}\label{e:setone}
A_p:=(1-t) \cdot_p A_s(f) +_p t \cdot_p A_s(g) \subset A_s(h_{p,t,s}) \subset A_s(h),
\end{equation}
where $h_{p,t,s}$ is the $L_{p,s}$--supremal convolution introduced in equation \eqref{e:newsupcolvolution}. 
\end{lemma}

\begin{proof}
The second inclusion appearing in \eqref{e:setone} follows immediately as $h \geq h_{p,t,s}$ by assumption \eqref{e:integerLpBBLassumption0}.  Therefore, we need only to establish the first inclusion.  Let $z \in A_p$.  Then $z$ is of the form 
\[
z = (1-t)^{1/p}(1-\lambda)^{(p-1)/p}z_1 + t^{1/p}\lambda^{(p-1)/p}z_2
\]
for some $z_1 \in A_s(f)$, $z_2 \in A_s(g)$ and $\lambda \in [0,1]$. 
Write $z_1 = (x_1,y_1)$ with $x_1 \in \text{supp}(f)$ and $y_1 \in \R^s$ satisfying $\|y_1\| \leq f(x)^{1/s}$. Similarly, $z_2 = (x_2,y_2)$ with $x_2 \in \text{supp}(g)$ and $y_2 \in \R^s$ and $\|y_2\| \leq g(x_2)^{1/s}$. Then, we have 
\[
z=((1-t)^{1/p}(1-\lambda)^{(p-1)/p}x_1+t^{1/p}\lambda^{(p-1)/p}x_2,(1-t)^{1/p}(1-\lambda)^{(p-1)/p}y_1+t^{1/p}\lambda^{(p-1)/p}y_2)= :(\bar{z},\tilde{z})
\]
where $\bar{z}=(1-t)^{1/p}(1-\lambda)^{(p-1)/p}x_1+t^{1/p}\lambda^{(p-1)/p}x_2$ belongs to $(1-t) \cdot_p \text{supp}(f) +_p t \cdot_p \text{supp}(g)$, and with
\begin{align*}
\|\tilde{z}\|&= \|(1-t)^{1/p}(1-\lambda)^{(p-1)/p}y_1+t^{1/p}\lambda^{(p-1)/p}y_2\| \\
 &\leq (1-t)^{1/p}(1-\lambda)^{(p-1)/p} \|y_1\| + t^{1/p}\lambda^{(p-1)/p}\|y_2\|\\
&\leq (1-t)^{1/p}(1-\lambda)^{(p-1)/p} f(x_1)^{1/s} + t^{1/p}\lambda^{(p-1)/p}g(x_2)^{1/s}\\
&\leq h_{p,t,s}((1-t)^{1/p}(1-\lambda)^{(p-1)/p}x_1+t^{1/p}\lambda^{(p-1)/p}x_2)^{1/s} \\
&= h_{p,t,s}(\bar{z})^{1/s},
\end{align*}
which yields the inclusion; i.e., for any $z\in\R^n$, $z\in A_s(h_{p,t,s})$ as long as $z\in A_p$, as desired.
\end{proof}
By the lemmas given above, we obtain the $L_p$-BBL inequality for functions when $s\in\N$ as follows.

\emph{Proof of Proposition \ref{t:integerlpbbl0}.} Let $f,g,h$ satisfy the condition \eqref{e:integerLpBBLassumption0}. Then, combining the set inclusions \eqref{e:setone} appearing in Lemma~\ref{t:setinclusion1} with the $L_p$-BMI \eqref{e:LYZBM}, we conclude that 
\begin{equation}\label{e:one}
\begin{split}
|A_s(h)|_{n+s} &\geq |A_s(h_{p,t,s})|_{n+s} \geq M_{\frac{p}{n+s}}^t(|A_s(f)|_{n+s},|A_s(g)|_{n+s}).
\end{split}
\end{equation}
Finally, combining the following equalities
\[
|A_s(h)|_{n+s} = |B_2^s|_s \int_{\R^n} h, \quad |A_s(h_{p,t,s})|_{n+s}= |B_2^s|_s \int_{\R^n} h_{p,t,s}, 
\]
with
\[
|A_s(f)|_{n+s} = |B_2^s|_s \int_{\R^n} f, \quad |A_s(g))|_{n+s}= |B_2^s|_s \int_{\R^n} g,
\]
 the inequality \eqref{e:one} directly yields inequality \eqref{e:integerLPBBL}, as desired.

The next goal is to establish $L_p$-BBL inequality for functions in the case when $s = \ell/m$ is a positive rational number in lowest terms for $\ell,m \in \N$. Given any integrable function $w \colon \R^n \to \R_+$ and positive integer $m$, consider the multiple function $\tilde{w} \colon (\R^n)^m \to \R_+$ defined as the product of $m$-independent copies of $w$, i.e., 
\[
\tilde{w}(x) = \tilde{w}(x_1,\dots,x_m) := \prod_{i=1}^m w(x_i).
\]
Using the independent structure of the multiple function $\tilde{w}$, Fubini's theorem yields 
\begin{equation}\label{e:integralvolumeequality}
\int_{(\R^n)^m} \tilde{w}(x)dx = \left(\int_{\R^n} w(x) dx\right)^m.
\end{equation}
Moreover, we see that 
\[
\text{supp}(\tilde{w}) = \text{supp}(w)\times \cdots \times \text{supp}(w)=:\text{supp}(w)^m,
\]
where, for a subset $A$ of $\R^n$,  $A^m = A \times \cdots \times A \subset (\R^n)^m$ denotes the Cartesian product of $m$ copies of $A$.

We require a few additional lemmas before returning to the proof of $L_p$-Borell-Brascamp-Lieb inequality for functions (Theorem~\ref{t:LPBBL}), the first of which establishes a relation between the $L_p$-Minkowski convex combination and Cartesian product as follows. 

\begin{lemma}\label{t:cartesianproducts} Let $p \in[1,\infty)$, $m \in \N$, and $\alpha, \beta >0$. Then, for any Borel sets $A,B \subset \R^n$, one has 
\[
[\alpha\cdot_pA +_p \beta \cdot_p B)]^m \supset \alpha \cdot_p A^m +_p \beta \cdot_p B^m. 
\]

\end{lemma}

\begin{proof}For each $\lambda \in [0,1]$, set $\alpha_{\lambda,p}:= \alpha^{1/p}(1-\lambda)^{(p-1)/p}$ and $\beta_{\lambda, p}:= \beta^{1/p}\lambda^{(p-1)/p}$.  Observe by the definition of $L_p$-Minkowski convex combination (\ref{e:LYZMinkowskicombo}) of the Cartesian product that
\begin{eqnarray*}
	&&\alpha \cdot_p A^m +_p \beta \cdot_p B^m \\
	&=& \{\alpha_{\lambda,p}\bar{x}+\beta_{\lambda,p} \bar{y} \colon \bar{x}\in A^m, \bar{y}\in B^m, 0 \leq \lambda \leq 1\}\\
	&=& \{\alpha_{\lambda,p}(x_1,\dots, x_m) + \beta_{\lambda,p}(y_1,\dots, y_m) \colon x_i \in A, y_i \in B, \text{ for all } 1=1,\dots, m, 0 \leq \lambda \ \leq1 \}\\
	&=& \{(\alpha_{\lambda,p} x_1 + \beta_{\lambda,p} y_1, \dots, \alpha_{\lambda,p} x_m+ \beta_{\lambda,p} x_m) \colon x_i \in A, y_i \in B, \text{ for all } 1=1,\dots, m, 0 \leq \lambda \leq 1\}\\
	&\subset&[\alpha\cdot_pA +_p \beta \cdot_p B)]^m,
\end{eqnarray*}
as desired.
\end{proof}

Our next key lemmas concerns, in some sense, the concavity conditions of the multiple function $\tilde{w}$ defined above.  Before we state the lemma, we will need the following version of H\"older's inequality (cf. \cite{Hardy}). 

\begin{lemma}\label{t:Holder}
Given $m \in \N$ and sequences of real numbers $\{a_1,\dots,a_m\}$ and $\{b_1,\dots,b_m\}$, one has 
\[
\left|\prod_{i=1}^m a_i \right| + \left| \prod_{i=1}^m b_i \right| \leq \left[\prod_{i=1}^m (|a_i|^m + |b_i|^m) \right]^{1/m}.
\]
\end{lemma}

As a result of the previous lemma, we obtain the following inequality of the $L_{p,s}$--supremal convolution of Cartesian product for functions.

\begin{corollary}\label{t:coro} Let $t,\lambda \in [0,1]$, $p \in[1,\infty)$, and $s = \ell/m$, with $\ell,m \in \N$ in lowest terms. Given $u,w \colon \R^n \to \R_+$, and $(x_1,\dots,x_m),(y_1,\dots,y_m)\in(\R^n)^m$, one has 
\begin{align*}
\prod_{i=1}^m \left[(1-t)^{1/p}(1-\lambda)^{(p-1)/p} u(x_i)^{1/s} + t^{1/p}\lambda^{(p-1)/p} w(y_i)^{1/s} \right]^{1/m}\\
\geq  (1-t)^{1/p}(1-\lambda)^{(p-1)/p} \prod_{i=1}^m u(x_i)^{1/\ell} + t^{1/p}\lambda^{(p-1)/p} \prod_{i=1}^m w(y_i)^{1/\ell}.
\end{align*}

\end{corollary}

\begin{proof}
The results follows immediately by taking the sequences $(a_1,\dots,a_m)$ and $(b_1,\dots,b_m)$, with 
\[
a_i = [(1-t)^{1/p}(1-\lambda)^{(p-1)/p}]^{1/m} u(x_i)^{1/\ell}, \quad b_i = [t^{1/p}\lambda^{(p-1)/p}]^{1/m} w(y_i)^{1/\ell}, \quad i=1,\dots,m
\]
in Lemma~\ref{t:Holder}.
\end{proof}

Set $s = \ell/m$, with $\ell,m \in \N$ in lowest terms. Next, we consider bodies of revolution akin to those appearing in equation \eqref{e:integerbodyofrev}. Given any bounded integrable function $w \colon \R^n \to \R_+$ that is not identically zero, define
\begin{eqnarray}\label{e:bodyofrev2}
B_s(w) \nonumber&:=& A_s(\tilde{w}) =\{(x,y) \in (\R^n)^m\times\R^{\ell} \colon x\in \text{supp}(\tilde{w}), \|y\| \leq \tilde{w}(x)^{1/\ell}\}\\
&=&\left\{(x_1,\dots,x_m,y) \in (\R^n)^m \times \R^{\ell} \colon x_i \in \text{supp}(w), i=1,\dots,m, \|y\| \leq \prod_{i=1}^m w(x_i)^{1/\ell}\right\}.
\end{eqnarray}
Notice that $B_s(w)$ is a $nm+l$ dimensional convex body if and only if the function $\tilde{w}$ is $(1/\ell)$-concave on its support. Moreover, we have 
\begin{equation}\label{e:integralequality2}
|B_s(w)|_{nm+\ell} = |B_2^{\ell}|_{\ell}\left(\int_{\R^n} w(x)dx\right)^m.
\end{equation}
Similar to Lemma \ref{t:setinclusion1}, we obtain the following set inclusion relation of the $nm+l$ dimensional body of revolution with respect to the $L_{p,s}$--supremal convolution for functions.
\begin{lemma}\label{t:anotherinclusion} Let $p \in[1,\infty)$, $s=\ell/m$, with $\ell,m \in \N$, and $t\in [0,1]$.  Let $f,g,h \colon \R^n \to \R_+$ be a triple satisfying the inequality (\ref{e:LpBBLassumption0}), and let $h_{p,t,s}$ denote the $L_{p,s}$--supremal convolution of $f$ and $g$.  Then
\begin{equation}\label{e:anothersetinclusion}
B_p :=(1-t) \cdot_p B_s(f) +_p t \cdot_p B_s(g)\subset B_s(h_{p,t,s}) \subset B_s(h). 
\end{equation}

\end{lemma}

\begin{proof}
The second inclusion appear in \eqref{e:anothersetinclusion} follows immediately since $h \geq h_{p,t,s}$. Therefore, it is enough to show that $B_p \subset B_s(h_{p,t,s})$. 

We begin by remarking that $B_p$ is the set of all $(x,y)\in (\R^n)^m \times \R^{\ell}$ such that 
\[
x \in (1-t)\cdot_p\text{supp}(\tilde{f}) +_p t \cdot_p \text{supp}(\tilde{g})
\]
and 
\begin{equation}\label{e:multipledimensional1}
\|y\| \leq \sup\left\{ (1-t)^{1/p}(1-\lambda)^{(p-1)/p}\tilde{f}(x_1) + t^{1/p}\lambda^{(p-1)/p}\tilde{g}(x_2)\right\}^{1/\ell},
\end{equation}
where the supremum is taken twice: once over all ways to write $x=(1-t)^{1/p}(1-\lambda)^{(p-1)/p}x_1 + t^{1/p}\lambda^{(p-1)/p}x_2$ and once over all $0 \leq \lambda \leq 1$. That is, $B_p = A_{\ell}(\widetilde{h_{p,t,s}})$. 

Now, using definition \eqref{e:bodyofrev2} together with Lemma~\ref{t:cartesianproducts}, we see that 
\begin{eqnarray}\label{e:multipledimensional2}
&&\!\!\!\!\!\! \!\!\!\!\!\!\!\!\!\!\!\!  B_s(h_{p,t,s}) \nonumber\\&=&\nonumber
\left\{(z,y) \in (\R^n)^m\times \R^{\ell} \colon z \in \text{supp}(\widetilde{h_{p,t,s}}), \|y\| \leq \widetilde{h_{p,t,s}}(z)^{1/\ell}\right\}\\
&=&\left\{(z,y) \in (\R^n)^m\times \R^{\ell} \colon  z \in [(1-t) \cdot_p \text{supp}(f) +_p t \cdot_p \text{supp}(g)]^m, \|y\| \leq \widetilde{h_{p,t,s}}(z)^{1/\ell}  \right\}\label{e:multipledimensional2}\\
&\supset&\nonumber\left\{(z,y) \in (\R^n)^m\times \R^{\ell} \colon  z \in (1-t) \cdot_p \text{supp}(\tilde{f}) +_p t \cdot_p \text{supp}(\tilde{g}), \|y\| \leq \widetilde{h_{p,t,s}}(z)^{1/\ell} \right\}.
\end{eqnarray}
Therefore, to complete the proof, it is enough to compare the last set inclusion relation appearing in \eqref{e:multipledimensional2} with the condition appearing in \eqref{e:multipledimensional1}. For every $z \in (1-t) \cdot_p \text{supp}(\tilde{f}) +_p t \cdot_p \text{supp}(\tilde{g})$ consider 
\begin{eqnarray*}
	&&\!\!\!\!\!\! \!\!\!\!\!\! \!\!\!\!\!\! \sup \{(1-t)^{1/p}(1-\lambda)^{(p-1)/p} \tilde{f}(x)^{1/\ell} + t^{1/p}\lambda^{(p-1)/p}\tilde{g}(y)^{1/\ell}\}\\
	&=& \sup \left\{(1-t)^{1/p}(1-\lambda)^{(p-1)/p} \prod_{i=1}^m f(x_i)^{1/\ell} +t^{1/p}\lambda^{(p-1)/p} \prod_{i=1}^m g(y_i)^{1/\ell}  \right\},
\end{eqnarray*}
where the supremum is taken twice: once over all ways to write 
\[
z = (1-t)^{1/p}(1-\lambda)^{(p-1)/p}x+t^{1/p}\lambda^{(p-1)/p}y,
\]
with $x \in \text{supp}(\tilde{f})$ and $y\in \text{supp}(\tilde{g})$ and once over all $0 \leq \lambda \leq 1$.  With this, Corollary~\ref{t:coro} implies that 
\begin{eqnarray*}
	&&\!\!\!\!\!\! \!\!\!\!\!\! \!\!\!\!\!\! \sup \{(1-t)^{1/p}(1-\lambda)^{(p-1)/p} \tilde{f}(x)^{1/\ell} + t^{1/p}\lambda^{(p-1)/p}\tilde{g}(y)^{1/\ell}\}\\&
	\leq&\sup\left\{\prod_{i=1}^m \left[(1-t)^{1/p}(1-\lambda)^{(p-1)/p} f(x_i)^{1/s} + t^{1/p}\lambda^{(p-1)/p} g(y_i)^{1/s} \right]^{1/m} \right\}\\
	&\leq& \prod_{i=1}^m \sup \left\{ \left[(1-t)^{1/p}(1-\lambda)^{(p-1)/p} f(x_i)^{1/s} + t^{1/p}\lambda^{(p-1)/p} g(y_i)^{1/s} \right]^{1/m}\right\}\\
	&=&\prod_{i=1}^mh_{p,t,s}((1-t)^{1/p}(1-\lambda)^{(p-1)/p}x+t^{1/p}\lambda^{(p-1)/p}y)^{1/(sm)}\\
	&=&\widetilde{h_{p,t,s}}(z)^{1/\ell}.
\end{eqnarray*}
Consequently, 
\[
\|y\| \leq \sup\left\{ (1-t)^{1/p}(1-\lambda)^{(p-1)/p}\tilde{f}(x_1) + t^{1/p}\lambda^{(p-1)/p}\tilde{g}(x_2)^{1/\ell}\right\},
\]
that is, if $(x,y) \in B_p$, then $\|y\| \leq \widetilde{h_{p,t,s}}(x)^{1/\ell}$, so that $(x,y) \in B_s(h_{p,t,s})$, as desired.
\end{proof}
\begin{theorem}	[$L_p$-BBL inequality for functions]\label{t:LPBBL}

	Let $p \in [1,\infty)$, $0 \leq s < \infty$, and $t \in [0,1]$. Let $f,g,h \colon \R^n \to \R_+$ be a triple of bounded integrable functions. Suppose, in addition, that this triple satisfies the condition
	\begin{equation}\label{e:LpBBLassumption}
	h\left((1-t)^{\frac{1}{p}}(1-\lambda)^{\frac{p-1}{p}}x + t^{\frac{1}{p}}\lambda^{\frac{p-1}{p}}y\right) \geq \left[(1-t)^{\frac{1}{p}}(1-\lambda)^{\frac{p-1}{p}} f(x)^{\frac{1}{s}} + t^{\frac{1}{p}}\lambda^{\frac{p-1}{p}}g(y)^{\frac{1}{s}}\right]^s
	\end{equation}
	for every $x \in \text{supp}(f)$, $y\in \text{supp}(g)$ and every $\lambda \in [0,1]$. Then the following integral inequality holds:
	\begin{equation}\label{e:LPBBLinequality2}
	\int_{\R^n}h(x)dx \geq M_{\frac{p}{n+s}}^t\left(\int_{\R^n}f(x)dx,\int_{\R^n}g(x)dx \right).
	\end{equation}
\end{theorem}
\begin{proof} Let $s = \ell/m$ be a positive rational number in lowest terms and consider the bodies $B_s(f), B_s(g), B_s(h)$.  By applying Lemma~\ref{t:anotherinclusion}, inequality \eqref{e:LYZBM}, and using formula \eqref{e:integralequality2} we see that 
\begin{align*}
|B_2^{\ell}|_{\ell}\left(\int_{\R^n}h(x)dx\right)^{m} &= |B_s(h)|_{nm+\ell}\\
&\geq M_{\frac{p}{nm+\ell}}\left(|B_s(f)|_{nm+\ell},|B_s(g)|_{nm+\ell} \right)\\
&=|B_2^{\ell}|_{\ell} M_{\frac{pm}{nm+\ell}}\left(\int_{\R^n} f(x) dx, \int_{\R^n} g(x)dx \right)\\
&= |B_2^{\ell}|_{\ell}M_{\frac{p}{n+s}}\left(\int_{\R^n} f(x) dx, \int_{\R^n} g(x)dx \right),
\end{align*}
which is the desired result.  The $L_p$-BBL inequality for functions of general real number $s \geq 0$ follows from standard approximation of the results for rational number sequence. 
\end{proof}

\section{$L_p$-Pr\'ekopa-Leindler type inequality for product measures with quasi-concave densities}\label{pplleindlerienquality}

In this section we prove $L_p$-Pr\'ekopa-Leindler type inequality ($L_p$-PLI) for product measures with quasi-concave densities (Theorem~\ref{t:LpPrekopaLeindler0} or \ref{t:LpPrekopaLeindler}).  The following lemma with respect to the $L_p$-BMI for measures of quasi-concave densities is needed first.

\begin{lemma}\label{t:lpBmonedimensional}
Let $t,\lambda \in [0,1]$, $p\geq 1$ and $\mu$ be a measure on $\R$ having a quasi-concave density $\phi \colon \R \to \R_+$ with maximum at origin. Then, for any measurable sets $A,B \subset \R^n$, such that $(1-t)^{\frac{1}{p}}(1-\lambda)^{\frac{p-1}{p}} A + t^{\frac{1}{p}}\lambda^{\frac{p-1}{p}}B$ is also measurable, one has 
\begin{equation}\label{e:lpBmonedimensional}
\mu\left((1-t)^{\frac{1}{p}}(1-\lambda)^{\frac{p-1}{p}} A + t^{\frac{1}{p}}\lambda^{\frac{p-1}{p}}B\right) \geq (1-t)^{\frac{1}{p}}(1-\lambda)^{\frac{p-1}{p}} \mu(A) + t^{\frac{1}{p}}\lambda^{\frac{p-1}{p}}\mu(B).
\end{equation}
In particular,
\[
\mu((1-t) \cdot_p A +_p t \cdot_p B) \geq \sup_{0 \leq \lambda \leq 1} \left[(1-t)^{\frac{1}{p}}(1-\lambda)^{\frac{p-1}{p}} \mu(A) + t^{\frac{1}{p}}\lambda^{\frac{p-1}{p}}\mu(B) \right].
\]
\end{lemma}

\begin{proof} For $r \in [0,1]$, denote $C_r(\phi) = \{x\in\R \colon \phi(x) \geq r \|\phi\|_{\infty}\}$.  We begin by showing the set inclusion relation as
\begin{eqnarray*}
&&\!\!\!\!\!\!\!\!\!\!\!\!\!\!\!\left[(1-t)^{\frac{1}{p}}(1-\lambda)^{\frac{p-1}{p}}A + t^{\frac{1}{p}}\lambda^{\frac{p-1}{p}}B\right] \cap C_r(\phi)\\
&\supset& (1-t)^{\frac{1}{p}}(1-\lambda)^{\frac{p-1}{p}}[A \cap C_r(\phi)]+t^{\frac{1}{p}}\lambda^{\frac{p-1}{p}}[B \cap C_r(\phi)]
\end{eqnarray*}
for any measurable sets $A$ and $B.$

Let $z \in(1-t)^{\frac{1}{p}}(1-\lambda)^{\frac{p-1}{p}}[A \cap C_r(\phi)]+t^{\frac{1}{p}}\lambda^{\frac{p-1}{p}}[B \cap C_r(\phi)]$.  Then $z$ is of the form 
\[
z =(1-t)^{\frac{1}{p}}(1-\lambda)^{\frac{p-1}{p}}x + t^{\frac{1}{p}}\lambda^{\frac{p-1}{p}}y
\]
for some $x \in A \cap C_r(\phi)$ and $y \in B \cap C_r(\phi)$. It is obvious that $z \in (1-t)^{\frac{1}{p}}(1-\lambda)^{\frac{p-1}{p}}A + t^{\frac{1}{p}}\lambda^{\frac{p-1}{p}}B$. Using the quasi-concavity of $\phi$, we see that $(1-t)^{\frac{1}{p}}(1-\lambda)^{\frac{p-1}{p}}x, t^{\frac{1}{p}}\lambda^{\frac{p-1}{p}}y \in C_r(\phi)$ if $x,y\in C_r(\phi) $. Again by the  definition of $C_r(\phi)$, one has
\begin{align*}
 \phi(z) &=\phi\Big((1-t)^{\frac{1}{p}}(1-\lambda)^{\frac{p-1}{p}}x + t^{\frac{1}{p}}\lambda^{\frac{p-1}{p}}y\Big)\\
 &=\phi\left\{(1-t)^{\frac{1}{p}}(1-\lambda)^{\frac{p-1}{p}}x+ \left[1-(1-t)^{\frac{1}{p}}(1-\lambda)^{\frac{p-1}{p}}\right]\left(\frac{t^{\frac{1}{p}}\lambda^{\frac{p-1}{p}}}{1-(1-t)^{\frac{1}{p}}(1-\lambda)^{\frac{p-1}{p}}} \cdot y\right)\right\}\\
  &\geq
 \min_{x\in A \cap C_r(\phi) , y\in B \cap C_r(\phi)}\left\{\phi(x), \phi\left(\frac{t^{\frac{1}{p}}\lambda^{\frac{p-1}{p}}}{1-(1-t)^{\frac{1}{p}}(1-\lambda)^{\frac{p-1}{p}}} \cdot y\right)\right\}\\
  &\geq
\min_{x\in A \cap C_r(\phi) , y\in B \cap C_r(\phi)}\left\{\phi(x), \phi\left( y\right)\right\}\\
 &\geq r\|\phi\|_{\infty},
\end{align*}
where we have used the fact that   
\[
\frac{t^{\frac{1}{p}}\lambda^{\frac{p-1}{p}}}{1-(1-t)^{\frac{1}{p}}(1-\lambda)^{\frac{p-1}{p}}} \leq 1
\]
from H\"older's inequality and the fact that 
$\phi$ is quasi-concave with maximum at the origin.

It follows from Fubini's theorem and the one-dimensional Brunn-Minkowski inequality that 
\begin{eqnarray*}
&&\!\!\!\!\!\!\!\!\!\mu\left((1-t)^{\frac{1}{p}}(1-\lambda)^{\frac{p-1}{p}} A + t^{\frac{1}{p}}\lambda^{\frac{p-1}{p}}B\right)\\
&\geq& \|\phi\|_{\infty} \int_0^1 \left|[(1-t)^{\frac{1}{p}}(1-\lambda)^{\frac{p-1}{p}}A + t^{\frac{1}{p}}\lambda^{\frac{p-1}{p}}B] \cap C_r(\phi)\right|_1dr\\ 
&\geq&  \|\phi\|_{\infty} \int_0^1 \left|(1-t)^{\frac{1}{p}}(1-\lambda)^{\frac{p-1}{p}}[A \cap C_r(\phi)]+t^{\frac{1}{p}}\lambda^{\frac{p-1}{p}}[B \cap C_r(\phi)] \right|_1 dr\\
&\geq& (1-t)^{\frac{1}{p}}(1-\lambda)^{\frac{p-1}{p}} \mu(A) + t^{\frac{1}{p}}\lambda^{\frac{p-1}{p}}\mu(B),
\end{eqnarray*}
as claimed. Since $\lambda$ was selected arbitrarily, we have
\[
\mu((1-t) \cdot_p A +_p t \cdot_p B) \geq \sup_{0 \leq \lambda \leq 1} \left[(1-t)^{\frac{1}{p}}(1-\lambda)^{\frac{p-1}{p}} \mu(A) + t^{\frac{1}{p}}\lambda^{\frac{p-1}{p}}\mu(B) \right],
\]
as desired.
\end{proof}
By induction on the dimension $n$, the detailed proof of the $L_p$-PLI for product measures with quasi-concave densities are given as follows.
\begin{theorem}[$L_p$-PLI for product measures with quasi-concave densities] \label{t:LpPrekopaLeindler} Let $p > 1$, $t \in [0,1]$, and $\mu = \mu_1 \times \cdots \times \mu_n$ be a product measure on $\R^n$, where, for each $i = 1,\dots,n,$ $\mu_i$ is a measure on $\R$ having a quasi-concave density $\phi \colon \R \to \R_+$ with maximum at origin. Let $f,g,h \colon \R^n\to \R_+$ be a triple of measurable functions, with $f,g$ weakly unconditional and positively decreasing, that satisfy the condition
	\begin{equation}\label{e:lpprekopaleindlerassumption}
	h((1-t)^{1/p}(1-\lambda)^{(p-1)/p}x + t^{1/p}\lambda^{(p-1)/p}y) \geq f(x)^{(1-t)^{1/p}(1-\lambda)^{(p-1)/p}}g(y)^{t^{1/p}\lambda^{(p-1)/p}}
	\end{equation}
	for every $x \in \text{supp}(f), y \in \text{supp}(g)$, and every $0 < \lambda < 1$. The the following integral inequality holds:
	\begin{equation}\label{e:lpprekopaleindlerconclusion2}
	\int_{\R^n} h d\mu \geq \sup_{0 < \lambda < 1}\left\{\left[\left(\frac{1-t}{1-\lambda}\right) \left(\frac{t}{\lambda}\right)\right]^{\frac{n}{p}}\left(\int_{\R^n} f^{\left(\frac{1-t}{1-\lambda}\right)^{\frac{1}{p}}} d\mu \right)^{1-\lambda}\left(\int_{\R^n}g^{\left(\frac{t}{\lambda}\right)^{\frac{1}{p}}} d\mu \right)^{\lambda} \right\}.
	\end{equation}
	
\end{theorem}
\begin{proof}  Without loss of generality we may assume that $f,g$ are bounded with $0<\|f\|_{\infty},\|g\|_{\infty}<\infty$ and $t \in (0,1).$ Fix $\lambda \in (0,1)$ arbitrarily.   

Assume that $n =1$. Multiplying the assumptions and conclusion of the theorem by constants $c_f,c_g,c_h$ with $$c_h :=\sup_{0 < \lambda < 1} (c_f)^{(1-t)^{1/p}(1-\lambda)^{(p-1)/p}}(c_g)^{t^{1/p}\lambda^{(p-1)/p}},$$
by taking $c_f:= \|f\|_{\infty}^{-1}, c_g := \|g\|_{\infty}^{-1}$, and therefore
\[
c_h:=\sup_{0 < \lambda < 1} (\|f\|_{\infty})^{-(1-t)^{1/p}(1-\lambda)^{(p-1)/p}}(\|g\|_{\infty})^{-t^{1/p}\lambda^{(p-1)/p}}.
\]
We assume that $\|f\|_{\infty} = \|g\|_{\infty} = 1$ without loss of generality.  We claim that for any $0 \leq r \leq 1$, the level set has the following relation
\[\{h \geq r\} \supset (1-t)^{\frac{1}{p}}(1-\lambda)^{\frac{p-1}{p}} \left\{f^{\left(\frac{1-t}{1-\lambda}\right)^{\frac{1}{p}}} \geq r\right\} + t^{\frac{1}{p}} \lambda^{\frac{p-1}{p}}\left\{g^{\left(\frac{t}{\lambda}\right)^{\frac{1}{p}}}\geq r\right\}.
\]
 Indeed, for any  $x \in \left\{f^{\left(\frac{1-t}{1-\lambda}\right)^{\frac{1}{p}}} \geq r\right\}$ and
 $y \in \left\{g^{\left(\frac{t}{\lambda}\right)^{\frac{1}{p}}}\geq r\right\}$, the hypothesis \eqref{e:lpprekopaleindlerassumption} of the theorem implies that 
\begin{align*}
h((1-t)^{1/p}(1-\lambda)^{(p-1)/p}x + t^{1/p}\lambda^{(p-1)/p}x)&\geq f(x)^{(1-t)^{1/p}(1-\lambda)^{(p-1)/p}}g(y)^{t^{1/p}\lambda^{(p-1)/p}}\\
&\geq \left(f^{\left(\frac{1-t}{1-\lambda}\right)^{\frac{1}{p}}}(x) \right)^{1-\lambda} \left(g^{\left(\frac{t}{\lambda}\right)^{\frac{1}{p}}}(y) \right)^{\lambda}\\
&\geq r^{1-\lambda+\lambda} = r.
\end{align*}
Therefore, using Fubini's theorem together with inequality \eqref{e:lpBmonedimensional} of Lemma~\ref{t:lpBmonedimensional}, we see that 
\begin{align*}
\int_{\R}h d\mu &\geq \int_0^{\infty}\mu(\{h \geq r\}) dr\\
&\geq \int_0^1 \mu(\{h \geq r\}) dr\\
&\geq  (1-t)^{\frac{1}{p}}(1-\lambda)^{\frac{p-1}{p}} \int_0^1 \mu\left(\left\{f^{\left(\frac{1-t}{1-\lambda}\right)^{\frac{1}{p}}} \geq r\right\} \right) dr\\
&+t^{\frac{1}{p}}\lambda^{\frac{p-1}{p}}\int_0^1 \mu\left( \left\{g^{\left(\frac{t}{\lambda}\right)^{\frac{1}{p}}}\geq r\right\}\right) dr \\
&=  (1-\lambda)\left(\frac{1-t}{1-\lambda}\right)^{\frac{1}{p}}\left(\int_{\R} f^{\left(\frac{1-t}{1-\lambda}\right)^{\frac{1}{p}}}(x) d\mu(x) \right) + \lambda\left(\frac{t}{\lambda}\right)^{\frac{1}{p}}\left(\int_{\R}g^{\left(\frac{t}{\lambda}\right)^{\frac{1}{p}}}(x) d\mu(x) \right) \\
&\geq   \left[\left(\frac{1-t}{1-\lambda}\right)^{1-\lambda}\left(\frac{t}{\lambda} \right)^{\lambda}\right]^{\frac{1}{p}}\left( \int_{\R} f^{\left(\frac{1-t}{1-\lambda}\right)^{\frac{1}{p}}}(x) d\mu(x)\right)^{1-\lambda}\left(\int_{\R} g^{\left(\frac{t}{\lambda}\right)^{\frac{1}{p}}}(x) d \mu(x)\right)^{\lambda},
\end{align*}
where, in the last step, we have used the arithmetic-geometric means inequality. This completes the proof for the case when $n = 1$. 

Assume that the conclusion of the theorem holds in dimension $n-1$ for some $n \geq 2$. Let $\lambda \in [0,1]$, $y_0,y_1 \in \R$ be arbitrary.  Set $y_2 := (1-t)^{1/p}(1-\lambda)^{(p-1)/p}y_0 + t^{1/p}\lambda^{(p-1)/p}y_1$.  Consider the function $f_{y_0},g_{y_1},h_{y_2} \colon \R^{n-1} \to \R_+$ defined by 
\[
f_{y_0}(x) = f(y_0,x), \quad g_{y_1}(x) = g(y_1,x), \quad h_{y_2}(x) = h(y_2,x). 
\]
The functions $f_{y_0}$ and $g_{y_1}$ are weakly unconditional and positively decreasing. 

Let $x_0,x_1 \in \R^{n-1}$ be arbitrary points belonging to the supports of $f_{y_0}$ and $g_{y_1}$, respectively, and set $x_2 := (1-t)^{\frac{1}{p}}(1-\lambda)^{\frac{p-1}{p}}x_0 + t^{\frac{1}{p}}\lambda^{\frac{p-1}{p}}x_1$. Using the hypothesis \eqref{e:lpprekopaleindlerassumption} placed on the triple of functions $f,g,h$, we have that 
\begin{align*}
h_{y_2}(x_2) &= h((1-t)^{\frac{1}{p}}(1-\lambda)^{\frac{p-1}{p}}y_0 + t^{\frac{1}{p}}\lambda^{\frac{p-1}{p}}y_1,(1-t)^{\frac{1}{p}}(1-\lambda)^{\frac{p-1}{p}}x_0 + t^{\frac{1}{p}}\lambda^{\frac{p-1}{p}}x_1)\\
&\geq f(y_0,x_0)^{(1-t)^{\frac{1}{p}}(1-\lambda)^{\frac{p-1}{p}}}g(y_1,x_1)^{t^{\frac{1}{p}}\lambda^{\frac{p-1}{p}}} \\
&= f_{y_0}(x_0)^{(1-t)^{\frac{1}{p}}(1-\lambda)^{\frac{p-1}{p}}}g_{y_1}(x_1)^{t^{\frac{1}{p}}\lambda^{\frac{p-1}{p}}}.
\end{align*}
Consequently, as all quantities involved were selected arbitrarily, the triple of functions $f_{y_0},g_{y_1},h_{y_2}$ satisfy the hypothesis of the theorem in dimension $n-1$, and so by the inductive hypothesis, we see that 
\begin{align*}
\int_{\R^{n-1}} h_{y_2} d\bar{\mu} \geq \left[\left(\frac{1-t}{1-\lambda}\right)^{1-\lambda}\left(\frac{t}{\lambda} \right)^{\lambda}\right]^{(n-1)\cdot\frac{1}{p}}\left( \int_{\R^{n-1}} f_{y_0}^{\left(\frac{1-t}{1-\lambda}\right)^{\frac{1}{p}}}d\bar{\mu}\right)^{1-\lambda}\left(\int_{\R^{n-1}} g_{y_1}^{\left(\frac{t}{\lambda}\right)^{\frac{1}{p}}} d \bar{\mu}\right)^{\lambda},
\end{align*}
where $\bar{\mu} = \mu_{1}\times \cdots \times \mu_{n-1}.$ 

Now define the one dimensional functions $F,G,H \colon \R \to \R_+$ in the following
\[
H(y_2) := \left[\left(\frac{1-t}{1-\lambda}\right)^{1-\lambda} \left(\frac{t}{\lambda}\right)^{\lambda}\right]^{- \left((n-1) \cdot \frac{1}{p} \right)}\int_{\R^{n-1}}h_{y_2}(x) d\bar{\mu}(x), 
\]
\[
F(y_0) := \int_{\R^{n-1}} f_{y_0}^{\left(\frac{1-t}{1-\lambda}\right)^{\frac{1}{p}}}(x)d\bar{\mu}(x), \quad G(y_1) := \int_{\R^{n-1}}g_{y_1}^{\left(\frac{t}{\lambda}\right)^{\frac{1}{p}}}(x) d\bar{\mu}(x).
\]
It can be seen that $F$ and $G$ are weakly unconditional and positively decreasing functions as the functions $f_{y_0}^{\left(\frac{1-t}{1-\lambda}\right)^{\frac{1}{p}}}$ and $g_{y_1}^{\left(\frac{t}{\lambda}\right)^{\frac{1}{p}}}$ remain weakly unconditional and positively decreasing. 
Then it can be easily checked that
\begin{eqnarray*}
&&\!\!\!\!\!\!\!\!\!\!\!\!\!\!\!H\left((1-t)^{\frac{1}{p}}(1-\lambda)^{\frac{p-1}{p}}y_0+t^{\frac{1}{p}}\lambda^{\frac{p-1}{p}}y_1\right) \\
&\geq& F(y_0)^{1-\lambda}G(y_1)^{\lambda}\\
&=& F(y_0)^{\left((1-t)^{\frac{1}{p}}(1-\lambda)^{\frac{p-1}{p}}\right)\left(\frac{1-\lambda}{1-t}\right)^{\frac{1}{p}}}G(y_1)^{\left(t^{\frac{1}{p}}\lambda^{\frac{p-1}{p}}\right)\left(\frac{\lambda}{t}\right)^{\frac{1}{p}}}.
\end{eqnarray*}
Furthermore, applying the one dimensional case of the inequality, we see that 
\begin{align*}
    \int_{\R}H(y_2) d\mu_n(y_2) &\geq  \left[\left(\frac{1-t}{1-\lambda}\right)^{1-\lambda} \left(\frac{t}{\lambda}\right)^{\lambda}\right]^{\frac{1}{p}}\left(\int_{\R} \left(F(y_0)^{\left(\frac{1-\lambda}{1-t}\right)^{\frac{1}{p}}}\right)^{\left(\frac{1-t}{1-\lambda}\right)^{\frac{1}{p}}}d\mu_n(y_0)\right)^{1-\lambda}\\
    &\times \left(\int_{\R} \left(G(y_1)^{\left(\frac{\lambda}{t}\right)^{\frac{1}{p}}}\right)^{\left(\frac{t}{\lambda}\right)^{\frac{1}{p}}}d\mu_n(y_1)\right)^{\lambda}\\
\end{align*}

that is,
\begin{align*}
\int_{\R^n} h(x) d\mu(x) &= \int_{\R} \int_{\R^{n-1}} h_{y_2}(x) d\bar{\mu}(x) d\mu_n(y_2)\\
&\geq \left[\left(\frac{1-t}{1-\lambda}\right)^{1-\lambda} \left(\frac{t}{\lambda}\right)^{\lambda}\right]^{\frac{n}{p}}\left(\int_{\R^n}F(y_0) d\mu_n(y_0)\right)^{1-\lambda}\\
    &\times \left(\int_{\R} G(y_1)d\mu_n(y_1)\right)^{\lambda}\\
    &= \left[\left(\frac{1-t}{1-\lambda}\right)^{1-\lambda} \left(\frac{t}{\lambda}\right)^{\lambda}\right]^{\frac{n}{p}}\left(\int_{\R^n} f^{\left(\frac{1-t}{1-\lambda}\right)^{\frac{1}{p}}}(x) d\mu(x) \right)^{1-\lambda}\\
    &\times \left(\int_{\R^n} g^{\left(\frac{t}{\lambda}\right)^{\frac{1}{p}}}(x) d\mu(x) \right)^{\lambda}.
\end{align*}

Since all quantities involved were selected arbitrarily, the desired inequality follows. 

\end{proof}

We noticed by using a similar proof, we can obtain the following theorem for the Lebesgue measure $\mu$ in Theorem~\ref{t:LpPrekopaLeindler}  which removes the weak unconditional condition on the functions involved. 

\begin{corollary}
Let $p > 1$ and $t \in [0,1]. $ Let $f,g,h \colon \R^n\to \R_+$ be a triple of measurable functions that satisfy the condition
\[
h((1-t)^{1/p}(1-\lambda)^{(p-1)/p}x + t^{1/p}\lambda^{(p-1)/p}y) \geq f(x)^{(1-t)^{1/p}(1-\lambda)^{(p-1)/p}}g(y)^{t^{1/p}\lambda^{(p-1)/p}}
\]
for every $x \in \text{supp}(f), y \in \text{supp}(g)$, and every $0 < \lambda < 1$. The the following integral inequality holds:
\[
\int_{\R^n} h dx \geq \sup_{0 < \lambda < 1}\left\{\left[\left(\frac{1-t}{1-\lambda}\right) \left(\frac{t}{\lambda}\right)\right]^{\frac{n}{p}}\left(\int_{\R^n} f^{\left(\frac{1-t}{1-\lambda}\right)^{\frac{1}{p}}} dx \right)^{1-\lambda}\left(\int_{\R^n}g^{\left(\frac{t}{\lambda}\right)^{\frac{1}{p}}} dx \right)^{\lambda} \right\}.
\]

\end{corollary}
If $p=1$, then $\lambda=t$ and it is exactly the Pr\'ekopa-Leindler inequality; i.e.,
for \(0<t<1\) and $f, g, h: \R^{n} \rightarrow \R_{+}$ three non-negative measurable functions defined on $\R^n$,  if these functions satisfy
$$
h((1-t) x+t y) \geq f(x)^{1-t} g(y)^{t}
$$
for all \(x\) and \(y\) in \(\R^{n}\), then
$$
\int_{\mathbb{R}^{n}} h(x) \mathrm{d} x \geq\left(\int_{\mathbb{R}^{n}} f(x) \mathrm{d} x\right)^{1-t}\left(\int_{\mathbb{R}^{n}} g(x) \mathrm{d} x\right)^{t}.
$$

\section{$L_p$-Minkowski's first type and $L_p$ isoperimetric inequalities for measures }
Similar to the definition of mixed volume and surface area to be the derivative for the perturbation of the Minkowski combination for convex bodies, we will give the surface area for functions with respect to the $L_{p,s}$--supremal convolution in this section.
Furthermore, we will explore $L_p$ versions of the Minkowski's first inequality for non-negative  functions, which gives rise to $L_p$-Minkowski first type inequalities for certain classes of measures. 

Firstly, the definition of $F$-concavity in terms of the integral of $L_{p,s}$--supremal convolution for measures is defined in the following way.

\begin{definition}\label{definitionoffconcavity} Let $p \in [1,\infty)$, $s \in [0,\infty]$, and $F \colon \R_+ \to \R$ be an invertible differentiable function.  We say that a measure $\mu$ on $\R^n$ is $F(t)$-concave with respect to the $L_{p,s}$--supremal convolution of functions belonging to some class $\mathcal{A}$ of bounded non-negative $\mu$-integrable functions if, for any members $f,g$ belonging to $\mathcal{A}$ and any $t \in [0,1]$, one has that 
\begin{equation}\label{e:functionalMink1stassumption}
\int_{\R^n} [((1-t) \times_{p,s} f)\oplus_{p,s} (t \times_{p,s} g)] d\mu \geq F^{-1}\left((1-t)F\left(\int_{\R^n} f d\mu\right) + t F\left(\int_{\R^n} g d\mu\right)\right).
\end{equation}
\end{definition}

Inspired by the works of Colesanti and Fragal\'a in \cite{ColesantiFragala} and of Klartag in \cite{Klartag}, we give the following definition of the surface area extension for $L_{p,s}$--supremal convolution with measures.

\begin{definition} Let $\mu$ be a Borel measure on $\R^n$, $p \geq 1$, and $s \in [0, \infty]$. We define the $L_p$-$\mu$-surface area of a $\mu$-integrable function $f \colon \R^n \to \R_+$ with respect to a $\mu$-integrable function $g$ by 
\[
S_{\mu,p,s}(f,g) := \liminf_{\e \to 0^+} \frac{\int_{\R^n} f \oplus_{p,s} (\e \times_{p,s} g) d\mu -\int_{\R^n}f d\mu }{\e}.
\]
If $\mu$ is the Lebesgue measure on $\R^n$, we will simply denote it as $S_{p,s}$.
\end{definition}

One of the main results in this section is the following theorem with respect to $L_p$-Minkowski first inequality, which is a generalization of Theorem~3.8 in \cite{Liv} (see also Theorem~4.1 in \cite{Wu}) in the sense of the functional setting of $L_p$ sum.

\begin{theorem}[$L_p$-MFI for  measures]\label{t:functionalMink1st} Let $p \in[1,\infty)$, $s \in [0,\infty]$, and $F \colon \R_+\to \R$ be a differentiable invertible function. Let $\mu$ be a Borel measure on $\R^n$, and assume that $\mu$ is $F(t)$-concave with respect to some class, $\mathcal{A}$, of non-negative bounded $\mu$-integrable functions and the $L_{p,s}$--supremal convolution.
Then the following inequality holds for any members $f,g$ of the class $\mathcal{A}$: 
\begin{equation}\label{e:functionalMink1stconclusion}
S_{\mu,p,s}(f,g) \geq S_{\mu,p,s}(f,f) + \frac{F\left(\int_{\R^n}g d\mu\right) - F\left(\int_{\R^n} f d\mu\right)}{F'\left(\int_{\R^n} f d\mu \right)}.
\end{equation}
In particular, when $\int_{\R^n} f d\mu = \int_{\R^n} g d\mu$, we obtain the following isoperimetric type inequality: 
\[
S_{\mu,p,s}(f,g) \geq S_{\mu,p,s}(f,f). 
\]
\end{theorem}

\begin{proof} Let $f,g \in \mathcal{A}$. Since $f,g$ are $\mu$-integrable, without loss of generality, we may assume that $f,g$ are compactly supported. 

Assume, first that $s \neq 0, +\infty$. According to the assumption \eqref{e:functionalMink1stassumption}, for any $\e>0$ sufficiently small, we may write 
\begin{eqnarray}
&&\nonumber\int_{\R^n} [f \oplus_{p,s} (\e \times_{p,s} g)](x) d\mu(x) \\
&=&\nonumber \int_{\R^n} \left\{\left[(1-\e) \times_{p,s} \left(\frac{1}{1-\e} \times_{p,s}f(x)\right)\right] \oplus_{p,s} (\e \times_{p,s} g(x))\right\}d\mu(x)\\
&\geq& F^{-1}\left\{(1-\e) F\left(\frac{1}{1-\e} \times_{p,s}f(x)d\mu(x) \right) + \e F\left( \int_{\R^n} g(x) d\mu(x) \right)\right\}\label{concaveinequality}.
\end{eqnarray}
Define the function 
\begin{equation}\label{inversesumofF}
G_{F, \mu,s,p}(\e) := F^{-1}\left[(1-\e) F\left(\frac{1}{1-\e} \times_{p,s}f(x)d\mu(x) \right) + \e F\left( \int_{\R^n} g(x) d\mu(x)\right)\right].
\end{equation}
When $\e = 0$ the equality in (\ref{concaveinequality}) holds true obviously, and we have $G_{F,\mu,s,p}(0) = \int_{\R^n} f(x) d\mu(x)$. Therefore, we see that 
\begin{align*}
S_{\mu,p,s}(f,g) &=\liminf_{\e \to 0^+} \frac{\int_{\R^n} [f \oplus_{p,s} (\e \times_{p,s} g)](x) d\mu(x) -\int_{\R^n}f(x) d\mu(x)}{\e}\\
&\geq \liminf_{\e \to 0^+} \frac{G_{F,\mu,s,p}(\e) - G_{F,\mu,s,p}(0)}{\e} = G_{F,\mu,s,p}'(0).
\end{align*}
What remains is to complete $G_{F,\mu,s,p}'(0)$. 
Next, we notice that 
\begin{align*}
&\frac{d}{d\e}\left[\int_{\R^n} (\frac{1}{1-\e} \times_{p,s} f)(x) d\mu(x)\right] \biggr\rvert_{\e=0}\\
&= \lim_{\e \to 0^+} \frac{\int_{\R^n} [(1+\e+\e^2+\cdots) \times_{p,s} f](x) d\mu(x) - \int_{\R^n} f(x) d\mu(x)}{\e}\\
&= S_{\mu,p,s}(f,f). 
\end{align*}
Finally, using  the fact that, for any invertible differentiable function $F \colon \R_+ \to \R$, one has 
\[
F^{-1}(a)' = \frac{1}{F'(F^{-1}(a))},
\]
and by (\ref{inversesumofF}), there is 
\begin{align*}
S_{\mu,p,s}(f,g) &\geq G_{F,\mu,s,p}'(0)\\
&= \frac{(1-\e) S_{\mu,p,s}(f,f)  \cdot F'\left(\int_{\R^n}\frac{1}{1-\e} \times_{p,s} f(x) d\mu(x) \right) \mid_{\e=0} }{F'\left(\int_{\R^n} f d\mu \right)}\\
&+\frac{- F \left( \int_{\R^n}[(1-\e) \times_{p,s} f](x) d\mu(x)\right)\mid_{\e=0}+F\left(\int_{\R^n} g(x) d\mu(x)\right)}{F'\left(\int_{\R^n} f d\mu \right)}\\
&=S_{\mu,p,s}(f,f) + \frac{F\left(\int_{\R^n}g(x)d\mu(x)\right) -F\left(\int_{\R^n}f(x)d\mu(x) \right)}{F'\left(\int_{\R^n}f(x)d\mu(x)\right)},
\end{align*}
as desired.

The arguments for $s=0, \infty$ follow in the same manner and the proofs are omitted.
\end{proof}
In particular, we obtain a natural  corollary of Theorem~\ref{t:functionalMink1st} for Lebesgue measure $\mu$ and for $s=\infty$ respectively.

\begin{corollary} 
	(i) Let $p \in [1,\infty)$, $s \in [0,\infty)$. Then, for any bounded integrable functions $f,g \colon \R^n \to \R_+$, one has 
\[
S_{p,s}(f,g) \geq S_{p,s}(f,f) +\frac{\left(\int_{\R^n}g(x)dx \right)^{\frac{p}{n+s}}-\left(\int_{\R^n}f(x)dx \right)^{\frac{p}{n+s}}}{\frac{p}{n+s}\left(\int_{\R^n}f(x)dx\right)^{\frac{p}{n+s}-1}}.
\]
In particular, when $\int_{\R^n} f(x) dx = \int_{\R^n} g(x) dx>0$, we obtain the following isoperimetric type inequality:
\[
S_{p,s}(f,g) \geq S_{p,s}(f,f).
\]

(ii) Let $p \in [1,\infty)$. When $s =+\infty$, and $\mu$ is any $log$-concave measure on $\R^n$, the following inequality holds: 
\[
S_{\mu,p,\infty}(f,g) \geq S_{\mu,p,\infty}(f,f) + \left(\int_{\R^n} f(x) d\mu(x) \right) \log\left[\frac{\int_{\R^n} g(x) d\mu(x)}{\int_{\R^n} f(x) d\mu(x)} \right].
\]
If, in addition, $\int_{\R^n} f d\mu = \int_{\R^n} g d\mu > 0$, then we obtain the following isoperimetric type inequality: 
\[
S_{\mu,p,\infty}(f,g) \geq S_{\mu,p,\infty}(f,f).
\]
\end{corollary}

In Definition~\ref{definitionoffconcavity}, if one takes $f=1_{A}$, $g=1_B$, and $h_{1_{(1-t) \cdot_p A +_p t \cdot_p B}}$, then the definition goes back to  $F$-concavity with respect to the $L_p$-Minkowski convex combination for convex bodies  in \cite{Wu}.  For $p \geq 1$ and $F \colon \R_+ \to \R$ a strictly increasing invertible differentiable function,  we say that a non-negative measure $\mu$ on $\R^n$ (that is absolutely continuous with respect to the Lebesgue measure on $\R^n$) is $F(t)$-concave with respect to the $L_p$-Minkowski convex combination on some class of Borel sets if, for all Borel sets $A,B \subset \R^n$ belonging to this class, and every $t \in [0,1]$, one has 
\begin{equation}\label{e:F(t)concave}
F(\mu((1-t) \cdot_p A +_p t \cdot_p B)) \geq (1-t)F(\mu(A))+ tF(\mu(B)).
\end{equation}

Instead of  the definitions in \cite{LMNZ} (see also \cite{Liv, JesusManuel,Wu}) for measure $\mu$ with respect to convex bodies: we propose the definition with respect to the composite of the function $F$ and measure $\mu$---$F\circ\mu$: given a measure $\mu$ on $\R^n$ which is $F(t)$-concave with respect to some class of Borel sets and members $A,B \subset \R^n$ belonging to this class, we define the $L_p$-$\mu$-surface area of the set $A$ with respect to the set $B$---$V_{p,F}^{\mu}(A,B)$ as
\begin{equation}\label{mixedvolume}
V_{p,F}^{\mu}(A,B) := F'(1) \cdot \liminf_{\e\to 0^+} \frac{\mu(A+_p\e\cdot_pB) - \mu(A)}{t}, \end{equation}
and denote by
\begin{equation}\label{surfacearea}
M_{p,F}^{\mu}(A):= \frac{1}{F'(1)}\cdot \mu(A) -\frac{d}{d\e}^{-} \biggr\rvert_{\e=1} \mu(\e \cdot_p A).
\end{equation} 
One lemma about the concavity of $F\circ\mu$ with respect of the $L_p$-Minkowski convex combination for convex bodies is needed.
\begin{lemma}\label{t:extragoodcoro}
Let $p \geq 1$, $F \colon \R_+ \to \R$ be a strictly increasing differentiable function and $\mu$ be a measure on $\R^n$ that is absolutely continuous with respect to the Lebesgue measure. Suppose that $\mu$ is $F(t)$-concave with respect to some class $\mathcal{A} \subset \mathcal{K}_{(o)}^n$ and closed with respect to the $L_p$-Minkowski convex combination of its members.  Then, for any members $A,B\in\mathcal{A}$, the maps  
\[
\e \mapsto F(\mu(\e \cdot_p A)), \quad \e \mapsto F(\mu(A+_p \e \cdot_p B))
\]
are concave on $[0,\infty).$
\end{lemma}

\begin{proof} Let $t \in [0,1]$ and $\e_1,\e_2 \geq 0$.
Since all sets involved are convex bodes containing the origin, it means $A+_p [((1-t)\e_1 + t \e_2) \cdot_p B]\in\mathcal{K}_{(o)}^n$. Therefore, we may use the support function definition of the $L_p$-Minkowski convex combination. Notice that 
\begin{align*}
h_{A+_p [((1-t)\e_1 + t \e_2) \cdot_p B]}&=\left[h_A^p +[(1-t)\e_1+t\e_2]h_B^p \right]^{\frac{1}{p}}\\
&=\left[(1-t)\left(h_A^p + \e_1h_B^p\right)+ t \left(h_A^p + \e_2 h_B^p\right) \right]^{\frac{1}{p}}\\
&=\left[(1-t)\left(h_{A +_p \e_1 \cdot_p B}\right)^p + t\left(h_{A +_p \e_2\cdot_p B}\right)^p\right]^{\frac{1}{p}}\\
&=h_{(1-t) \cdot_p [A+_p \e_1 \cdot_p B] +t \cdot_p[A + \e_2 \cdot_p B]}.
\end{align*}
Therefore, we see that 
\[
A+_p ((1-t)\e_1 + t \e_2) \cdot_p B = (1-t) \cdot_p [A+_p \e_1 \cdot_p B] +t \cdot_p[A + \e_2 \cdot_p B].
\]
Using the fact that $\mu$ is $F(t)$-concave with respect to the given class of convex bodies containing the origin, we obtain 
\begin{align*}
F(\mu(A+_p ((1-t)\e_1 + t \e_2) \cdot_p B)) &= F(\mu((1-t) \cdot_p [A+_p \e_1 \cdot_p B] +t \cdot_p[A + \e_2 \cdot_p B]))\\
&\geq (1-t) F(\mu(A+_p \e_1 \cdot_p B))+ t F(\mu(A+_p \e_2 \cdot_p B)).
\end{align*}
In the same spirit, the proof of the other inequality assertion follows obviously. 
\end{proof}

Together with the definition of $L_p$-$\mu$-surface area (\ref{mixedvolume}), (\ref{surfacearea}) and Lemma \ref{t:extragoodcoro}, we are now prepared to establish the following isoperimetric type inequality. 

\begin{theorem}[$L_p$-ISMI for measures]\label{t:lpisoperimetricinequalitygeneral}
Let $p \geq 1$, $F \colon \R_+ \to \R$ be a strictly increasing differentiable function and $\mu$ be a measure on $\R^n$ that is absolutely continuous with respect to the Lebesgue measure. Suppose that $\mu$ is $F(t)$-concave with respect to some class $\mathcal{A}\subset\mathcal{K}_{(o)}^n$ and closed with respect to the $L_p$-Minkowski convex combination of its members.  Then for any  $A,B\in\mathcal{A}$  and such that $\mu((1-t) \cdot_p A +_p t \cdot_p B))$ is finite, one has 
\[
V_{p,F}^{\mu}(A,B) + F'(1) M_{p,F}^{\mu}(A) \geq \frac{F'(1)[F(\mu(B))-F(\mu(A))]}{F'(\mu(A))}+\mu(A),
\]
with equality only if and only if $A=B$.
\end{theorem}

\begin{proof} 
	The proof of the theorem follows from the ideas of the proof of the classical isoperimetric inequality (cf.  \cite[Theorem~7.2.1]{Schneider:CB2}) and \cite{LMNZ, JesusManuel}. 
Consider the function $f \colon [0,1] \to \R_+$ defined by 
\[
f(t) = F(\mu((1-t) \cdot_p A +_p t \cdot_p B))-\left[(1-t) F(\mu(A)) + t F(\mu(B))\right].
\]
Since $\mu$ is $F(t)$-concave with respect to a class of convex bodies containing the origin and closed with respect to the $L_p$-Minkowski convex combination, the functions $f$ is concave  and such that $f(0) = f(1) =0.$ Therefore, the right-derivative of $f$ at $t=0$ exists (cf. \cite[Theorem~23.1]{Rock}) and, moreover, 
\[
\frac{d^+}{dt}\biggr\rvert_{t=0}f(t) \geq 0, 
\]
with equality if and only if $f(t) = 0$ for all $t \in [0,1]$.

As
\[
\frac{d^+}{dt}\biggr\rvert_{t=0}f(t) =F'(\mu(A))\cdot \frac{d^+}{dt}\biggr\rvert_{t=0}\mu((1-t) \cdot_pA + t \cdot_p B) + F(\mu(A))-F(\mu(B)),
\]
it suffices only to compute the right derivative at $0$ of $\mu((1-t)\cdot_p A +_p t \cdot _p B)$. 

To this end, we notice that, by Lemma~\ref{t:extragoodcoro}, the one-sided derivative of $\mu(t \cdot_p A)$ and $\mu(A +_p t \cdot_p B)$ at $t = 1$ and $t =0$, respectively, exist. Set $g(r,s)=\mu(r \cdot_p(A+_p s \cdot_p B))$. Then we have 
\begin{align*}
\frac{d^+}{dt}\biggr\rvert_{t=0}\mu((1-t) \cdot_pA + t \cdot_p B) &= \frac{d^+}{dt}\biggr\rvert_{t=0} g\left(1-t, \frac{t}{1-t}\right)\\
&=-\frac{d^-}{dt}\biggr\rvert_{t=1}\mu(t \cdot_p A) + \frac{d^+}{dt}\biggr\rvert_{t=0}\mu(A+_p t \cdot_p B)\\
&= M_{p,F}^{\mu}(A) - \frac{1}{F'(1)}\mu(A) + \frac{1}{F'(1)}V_{p,F}^{\mu}(A,B).
\end{align*}
Therefore,
\[
\frac{d^+}{dt}\biggr\rvert_{t=0}f(t) = F'(\mu(A))\left[M_{p,F}^{\mu}(A)- \frac{1}{F'(1)}\mu(A) + \frac{1}{F'(1)}V_{p}^{\mu}(A,B)\right]+ F(\mu(A))-F(\mu(B)).
\]
Combing the above inequality together wit the fact that $\frac{d^+}{dt}|_{t=0}f(t) \geq 0$ yields the inequality part of the theorem. The equality follows from the characterization of the equality case of $\frac{d^+}{dt}\rvert_{t=0}f(t) = 0$ mentioned above. 
\end{proof}

We draw the following consequences of isoperimetric type inequality from Theorem~\ref{t:lpisoperimetricinequalitygeneral}.  The first follows by combining Theorem~\ref{t:lpisoperimetricinequalitygeneral} with $L_p$-BMI for measures with $(1/s)$-concave densities (Theorem~\ref{t:Lpsconcavemeasures})  by taking the class of all convex bodies $\mathcal{K}_{(o)}^n$ and the functions $F(t) = t^{\frac{p}{n+s}}$, when $0 \leq s < \infty$. And another is by taking the function $F(t) = t^{\frac{p}{n}}$ in Theorem~\ref{t:lpisoperimetricinequalitygeneral} and combing this with $L_p$-BMI for product measures with quasi-concave densities (Thoerem~\ref{t:LpPLmeasures}), we obtain the following inequalities.

\begin{corollary}
(i) Let $p \in[1,\infty)$, $t\in[0,1]$, and $s \in [0,\infty)$. Let $\mu$ be a measure given by $d\mu(x) = \phi(x) dx$, where $\phi \colon \R^n \to \R_+$ is a $\left(\frac{1}{s}\right)$-concave function on its support. Then, for any convex bodies $A,B \subset \R^n$ containing the origin such that $(1-t) \cdot_p A +_p t \cdot_p B$ has finite $\mu$-measure, one has 
\[
V_{p,t^{\frac{p}{n+s}}}^{\mu}(A,B) + \frac{p}{n+s}M_{p,t^{\frac{p}{n+s}}}^{\mu}(A) \geq \mu(B)^{\frac{p}{n+s}}\mu(A)^{1-\frac{p}{n+s}},
\]
with equality if and only if $A=B.$

(ii) Let $p > 1$, $t \in [0,1]$, and $\mu = \mu_1 \times \cdots \times \mu_n$ be a product measure on $\R^n$, where, for each $i = 1,\dots,n,$ $\mu_i$ is a measure on $\R$ having a quasi-concave density $\phi \colon \R \to \R_+$ with maximum at origin. Then, for any weakly unconditional convex bodies $A,B$ such that $(1-t) \cdot_p A +_p t \cdot_p B$ has finite $\mu$-measure, one has 
\[
V_{p, t^{\frac{p}{n}}}^{\mu}(A,B)+ \frac{p}{n}M_{p,t^{\frac{p}{n}}}^{\mu}(A) \geq\mu(B)^{\frac{p}{n}}\mu(A)^{1-\frac{p}{n}},
\]
with equality if and only if $A=B.$
\end{corollary}

\section{A functional counterpart of the Gardner-Zvavitch conjecture}

The study of Borell-Brascamp-Lieb type inequalities, and their connections to isoperimetric type problems lead to  the following problem, which can be seen as a functional counterpart of the Gardner-Zvavitch conjecture (cf. \cite{GZ}):

\begin{conjecture}[Gardner-Zvavitch Conjecture]
Let $\gamma_n$ denote the standard Gaussian measure on $\R^n$ having density 
\[
\phi(x) = \frac{1}{(2\pi)^{n/2}}e^{-x^2/2}.
\]
Is it true that for any convex bodies $K,L \in \mathcal{K}_{(o)}^n$, one has 
\[
\gamma_n((1-t) K+tL)^{\frac{1}{n}} \geq (1-t) \gamma_n(K)^{\frac{1}{n}} + t \gamma_n(L)^{\frac{1}{n}}?
\]
\end{conjecture}

The above problem was recently shown to be true in \cite{EM} when $K$ and $L$ are taken to be origin symmetric. However, it was shown in \cite{PT} that, in general, the above conjecture fails without this symmetry assumption.  It is curious to know whether or not for general convex bodies containing the origin, the Gardner-Zvavitch conjecture holds true up to some universal constant. This can be seen by the following Corollary~\ref{t:LpGZmeasures}.

Motivated by the advances on the Gardner-Zvavitch conjecture, we ask the following functional counterpart of this conjecture, and answer it in some special cases.

\begin{conjecture}[Functional $L_p$-Gardner-Zvavitch conjecture] Let $p \in [1,\infty).$ Let $f,g \colon \R^n \to \R_+$ be centered integrable functions, and $\mu$ be measure on $\R^n$, and $t \in (0,1)$. Assume that $h \colon \R^n \to \R_+$ is a measurable function such that 
\[
h((1-t)^{\frac{1}{p}}(1-\lambda)^{\frac{p-1}{p}}x + t^{\frac{1}{p}}\lambda^{\frac{p-1}{p}}y) \geq f(x)^{(1-t)^{\frac{1}{p}}(1-\lambda)^{\frac{p-1}{p}}}g(y)^{t^{\frac{1}{p}}\lambda^{\frac{p-1}{p}}}
\]
holds for all $x, y \in \R^n$ and for all $0<\lambda<1$.  Does there exist an absolute constant $C\geq 1$ such that following inequality hold:
\[
\int_{\R^n} h(x) d\mu(x) \geq \frac{1}{C^p}\left[(1-t)\left(\int_{\R^n} f(x) d\mu(x)\right)^{p/n} + t \left(\int_{\R^n}g(x)d\mu(x) \right)^{p/n} \right]^{n/p}?
\]

\end{conjecture}

We answer the above conjecture in the following particular case when $p=1$ and $C=1$ for product measures. 

\begin{theorem}\label{t:functionalgzsquares}
Let $\mu = \mu_1 \times \cdots \times \mu_n$ be a product measure, where $\mu_i$ is a measure on $\R$ having quasi-concave function $\phi$ such that $\phi(0) = \|\phi\|_{\infty}$ for all $i\in\{1,\cdots,n\}$, and let $t \in (0,1)$. Suppose that $f,g \colon \R^n \to \R_+$ are measurable functions with each having maximum at the origin, and satisfying $\|f\|_{\infty} = \|g\|_{\infty}$, and whose super-level sets can be written as products of Borel subsets of $\R$ containing the origin.  Then, for any measurable function $h \colon \R^n \to \R_+$ satisfying 
\begin{equation}\label{mincondition}
h((1-t)x + t y) \geq \min\{f(x),g(y)\}
\end{equation}
holds for every $x,y \in \R^n$ and any $t \in [0,1]$, and with maximum at the origin, one has
\[
\int_{\R^n} h(x) d\mu(x) \geq \left[(1-t)\left(\int_{\R^n} f(x) d\mu(x)\right)^{1/n} + t \left(\int_{\R^n}g(x)d\mu(x) \right)^{1/n} \right]^n.
\]
\end{theorem}

\begin{proof} By replacing $f$ with $f/\|f\|_{\infty}$ and $g$ with $g/\|g\|_{\infty}$, we  assume that $\|f\|_{\infty}=\|g\|_{\infty} = 1$ without loss of generality. 
Let $H_r, F_r, G_r$ denote the super level sets of $h,f,g$, respectively, that is.
\[
H_r:=\{x \colon h(x) \geq r\}, \quad F_r:= \{x \colon f(x) \geq r\}, \quad G_r:=\{x \colon g(x) \geq r\}.
\]
By the conditions that $f,g,h$ satisfy, one sees that, for all $0 \leq r < 1$,  
\[
H_r \supset (1-t) F_r + t G_r.
\]
Now, using the fact that each $F_r, G_r$ are coordinate boxes, one can write
\[
F_r = \prod_{i=1}^n I_{i,r} \quad G_r = \prod_{i=1}^n J_{i,r}
\]
for some Borel sets $I_{i,r},J_{i,r} \subset \R$ all containing the origin; consequently, 
\[
(1-t) F_r + t G_r = \prod_{i=1}^n [(1-t)I_{i,r} + t J_{i,r}]. 
\]
Using the fact that $\mu$ is a product measure together with Fubini's theorem, we have that 
\begin{align*}
\int_{\R^n} h(x) d\mu(x) &\geq \int_0^1 \mu(H_r) dr\\
&\geq \int_0^1 \mu((1-t)F_r + tG_r)dr\\
&= \int_0^1 \prod_{i=1}^n \mu_i((1-t)I_{i,r} + t J_{i,r})dr.
\end{align*}
Using Lemma~\ref{t:lpBmonedimensional}, we have that, for each $i=1,\dots,n$, 
\[
\mu_i((1-t)I_{i,r} + t J_{i,r}) \geq (1-t) \mu_i(I_{i,r}) + t \mu_i(J_{i,r}),
\]
and so 
\begin{align*}
\int_{\R^n} h(x) d\mu(x) &\geq  \int_0^1 \left(\prod_{i=1}^n [(1-t) \mu_i(I_{i,r}) + t \mu_i(J_{i,r})] \right) dr\\
&= \prod_{i=1}^n \left( (1-t) \int_0^1 \mu_i(I_{i,r}) dr+ t \int_0^1 \mu_i(J_{i,r}) dr \right),
\end{align*}
where, in the last step, we have used independence of the product. Finally, applying Minkowski's inequality, one has:
\begin{align*}
\left[ \int_{\R^n} h(x) d\mu(x) \right]^{1/n} &\geq \left[\prod_{i=1}^n \left( (1-t) \int_0^1 \mu_i(I_{i,r}) dr+ t \int_0^1 \mu_i(J_{i,r})dr \right) \right]^{1/n}\\
&\geq (1-t) \left(\prod_{i=1}^n \int_0^1 \mu_i(I_{i,r})dr \right)^{1/n} + t\left(\prod_{i=1}^n \int_0^1 \mu_i(J_{i,r})dr \right)^{1/n}\\
&=(1-t) \left( \int_{\R^n} f(x) d\mu(x) \right)^{1/n} + t \left( \int_{\R^n} g(x) d\mu(x) \right)^{1/n},
\end{align*}
as desired.
\end{proof}

In particular, due to the definition of the $L_{p,s}$--supremal convolution for $p\geq1$, we  obtain the following Corollaries by replacing (\ref{mincondition}) by (\ref{e:LpBBLassumption}) for $s=1$. 

\begin{corollary} Let $p \geq 1$. 
Let $\mu = \mu_1 \times \cdots \times \mu_n$ be a product measure, where $\mu_i$ is a measure on $\R$ having a  quasi-concave density $\phi$ such that $\phi(0) = \|\phi\|_{\infty}$, and let $t \in (0,1)$. Suppose that $f,g \colon \R^n \to \R_+$ are measurable functions with each having maximum at the origin, and such that $\|f\|_{\infty} = \|g\|_{\infty}$, and whose super-level sets can be written as a product of Borel subsets of $\R$ each containing the origin.  Then, for any measurable function $h \colon \R^n \to \R_+$ satisfying 
\[
h\left((1-t)^{1/p}(1-\lambda)^{(p-1)/p} x + t^{1/p}\lambda^{(p-1)/p}y\right) \geq (1-t)^{1/p}(1-\lambda)^{(p-1)/p}f(x) + t^{1/p}\lambda^{(p-1)/p}g(y)
\]
holds for every $x \in \text{supp}(f)$, $y \in \text{supp}(g)$, and every $0 \leq \lambda \leq 1$, and such that $h$ attains its maximum at the origin, one has 
\[
\int_{\R^n} h(x) d\mu(x) \geq \left[(1-t)\left(\int_{\R^n} f(x) d\mu(x)\right)^{1/n} + t \left(\int_{\R^n}g(x)d\mu(x) \right)^{1/n} \right]^n.
\]
\end{corollary}

Combining the above corollary with Theorem~\ref{t:functionalMink1st}, that is letting $F(t)=t^{1/n}$ with the $L_{p,s}$--supremal convolution,
we obtain the inequalities as follows. 

\begin{corollary}Let $p\in [1,\infty)$ and $s \in [0,\infty]$. Let $\mu = \mu_1 \times \cdots \times \mu_n$ be a product measure, where $\mu_i$ is a measure on $\R$ having a  quasi-concave density $\phi$ such that $\phi(0) = \|\phi\|_{\infty}$, and let $t \in (0,1)$. Suppose that $f,g \colon \R^n \to \R_+$ are measurable functions with each having maximum at the origin, and such that $\|f\|_{\infty} = \|g\|_{\infty}$, and whose super-level sets can be written as products of Borel subsets of $\R$ containing the origin. Then one has 
\[
S_{\mu,p,s}(f,g) \geq S_{\mu,p,s}(f,f) + \frac{\left(\int_{\R^n} g(x) d\mu(x) \right)^{1/n} - \left(\int_{\R^n} f(x) d\mu(x) \right)^{1/n}}{\frac{1}{n}\left(\int_{\R^n} f(x) d\mu(x) \right)^{\frac{1}{n}-1}}.
\]
If $\int_{\R^n} f(x)d\mu(x) = \int_{\R^n} g(x)d\mu(x)> 0$, we have the following isoperimetric type inequality
\[
S_{\mu,p,s}(f,g) \geq S_{\mu,p,s}(f,f).
\]
\end{corollary}

Finally, for $s=\infty$, we show a special $L_p$-BMI of the normalized $L_p$-sum of log-concave functions (Theorem \ref{t:lpfunctionalGZ}) as follows.
It concerns the following class of $\log$-concave functions on $\R^n$:
\[
\mathcal{L}^n = \left\{f \colon \R^n\to \R_+ \colon f(0) = \|f\|_{\infty}, 0 < \int f < \infty, f\ \text{is} \ \log\text{-concave} \right\},
\]
with a notion of convex bodies associated to $\log$-concave functions belonging to the class $\mathcal{L}^n$ as follows.  Let $q >0$ and $f \in \mathcal{L}^n$. Following \cite{Ball,Bobkov} we consider the following critical sets 
\[
K_q(f) =\left\{x \in \R^n \colon \left(\frac{1}{\|f\|_{\infty}}\int_0^{\infty} q f(rx) r^{q-1} dr \right)^{-\frac{1}{q}} \leq 1 \right\}.
\]
It was shown in \cite{Ball} that $K_m(f)$ is a convex body containing the origin for any $q >0$ and any $f \in \mathcal{L}^n$, and whose radial function $\rho_{K_q(f)} \colon \mathbb{S}^{n-1} \to \R_+$ is given by
\begin{equation}\label{e:radialfunctionBallBody}
\rho_{K_q(f)}(u) = \left(\frac{1}{\|f\|_{\infty}}\int_0^{\infty} q f(ru) r^{q-1} dr \right)^{\frac{1}{q}}.
\end{equation}
Moreover, when $q=n$ we see that: 
\begin{equation}\label{e:Ballvolume}
|K_n(f)|_n = \frac{1}{\|f\|_{\infty}}\int_{\R^n}f(x) dx.    
\end{equation}
Indeed, integrating in polar coordinates, we see that 
\begin{align*}
|K_n(f)|_n &=\int_{K_n(f)}dx = n |B_2^n|_n\int_{\mathbb{S}^{n-1}} \int_0^{\rho_{K_n(f)}}r^{n-1}dr du\\
&=|B_2^n|_n \int_{\mathbb{S}^{n-1}}\rho_{K_n(f)}(u)^n du\\
&= |B_2^n|_n \int_{\mathbb{S}^{n-1}} \frac{n}{\|f\|_{\infty}}\int_{0}^{\infty}f(ru)r^{n-1}dr du\\
&= \frac{1}{\|f\|_{\infty}} \int_{\R^n} f(x) dx,
\end{align*}
as claimed.

For $q >0$ and $f \in \mathcal{L}^n$, we consider the level-set 
\[
L_n(f) = \{x \in \R^n \colon f(x) \geq e^{-n}\|f\|_{\infty}\}. 
\]
We make use of the following lemma, originally due to Klartag and Milman in \cite{KlartagMilman}.

\begin{lemma}\label{t:KM}
Let $f \in \mathcal{L}^n$. The the following set inclusion holds 
\[
K_n(f) \subset L_n(f) \subset C \cdot K_n(f),
\]
where $C>1$ is some universal constant. 
\end{lemma}
Now we are prepared to prove the $L_p$-BMI of the normalized $L_p$-sum of log-concave functions based on the above lemma of set inclusion relation between the level sets $K_n(f)$ and $L_n(f)$.
\begin{theorem}\label{t:lpfunctionalGZ} 
	Let $p \geq 1, t \in [0,1]$, and $f,g \in \mathcal{L}^n$.  Then
	\begin{equation}\label{e:lpfuncxtionalGZ}
	\begin{split}
	\left(\frac{1}{\|(1-t) \cdot_{p,\infty}f \oplus_{p,\infty} t \cdot_{p,\infty} g\|_{\infty}}\ \int_{\R^n}(1-t) \cdot_{p,\infty}f \oplus_{p,\infty} t \cdot_{p,\infty} g dx\right)^{\frac{p}{n}}\geq\\
	\frac{1}{C^p}\cdot \left[(1-t) \left( \frac{1}{\|f\|_{\infty}}\int_{\R^n} f(x)dx\right)^{\frac{p}{n}} + t \left(\frac{1}{\|g\|_{\infty}}\int_{\R^n} g(x)dx\right)^{\frac{p}{n}}\right]
	\end{split}
	\end{equation}
	where $C>1$ is an absolute constant.
\end{theorem}
\begin{proof} Let $f,g \in \mathcal{L}^n$ and set 
\[
h = (1-t) \cdot_{p,\infty}f \oplus_{p,\infty} t \cdot_{p,\infty}g \in \mathcal{L}^n,
\]
and consider the bodies $K_n(f),K_n(g)$, and $K_n(h)$ associated to $f,g$, and $h$, respectively. We notice that, if we can show the inclusion 
\begin{equation}\label{e:awesomeinclusion}
C \cdot K_n(h) \supset (1-t) \cdot_p K_n(f) +_p t \cdot K_n(g),    
\end{equation}
then, by combining the identity \eqref{e:Ballvolume} together with the $L_p$-BMI \eqref{e:LPBBLinequality2}, it would imply that 
\begin{align*}
\left( \frac{1}{\|h\|_{\infty}} \int_{\R^n} h(x) dx \right)^{\frac{p}{n}} &= |K_n(h)|^{\frac{p}{n}}\\
&\geq \frac{1}{C^p} \left( (1-t)|K_n(f)|^{\frac{p}{n}} + t |K_n(g)|^{\frac{p}{n}} \right)\\
&=\frac{1}{C^p}\left[(1-t) \left(\frac{1}{\|f\|_{\infty}}\int_{\R^n} f(x) dx\right)^{\frac{p}{n}} + t \left(\frac{1}{\|g\|_{\infty}} \int_{\R^n} g(x) dx\right)^{\frac{p}{n}}\right].
\end{align*}
Therefore, we need only to establish the inclusion \eqref{e:awesomeinclusion}. 

Using Lemma~\ref{t:KM}, to establish the desired inclusion, it suffices to show that 
\[
(1-t) \cdot_p L_n(f) +_p t \cdot_p L_n(g) \subset L_n(h). 
\]
Set $\bar{f} = f /\|f\|_{\infty}$ and $\bar{g}=g / \|g\|_{\infty}$. Let $z \in (1-t) \cdot_p L_n(f) +_p t \cdot_p L_n(g)$. Then there exist $x \in \text{supp}(f)$, $y\in \text{supp}(g)$, and $0 \leq \lambda \leq 1$ such that 
\[
z = (1-t)^{\frac{1}{p}}(1-\lambda)^{\frac{p-1}{p}}x + t^{\frac{1}{p}}\lambda^{\frac{p-1}{p}}y, \quad \bar{f}(x) \geq e^{-n}, \quad  \bar{g}(y)  \geq e^{-n}. 
\]
Using these conditions, we see that 
\[
\bar{f}(x)^{(1-t)^{\frac{1}{p}}(1-\lambda)^{\frac{p-1}{p}}}\bar{g}(y)^{t^{\frac{1}{p}}\lambda^{\frac{p-1}{p}}} \geq  e^{-n\left[(1-t)^{\frac{1}{p}}(1-\lambda)^{\frac{p-1}{p}}+t^{\frac{1}{p}}\lambda^{\frac{p-1}{p}} \right]} \geq e^{-n},
\]
where, in the last step, we have used H\"older's inequality to conclude that 
\[
[(1-t)^{\frac{1}{p}}(1-\lambda)^{\frac{p-1}{p}}+t^{\frac{1}{p}}\lambda^{\frac{p-1}{p}} \leq 1.
\]
Therefore, we see that $z \in L_n((1-t) \cdot_{p,\infty} \bar{f} \oplus_{p,\infty} t \cdot_{p,\infty} \bar{g})$.

Finally, to complete the proof, we observe that,
\begin{align*}
 h(z) &= \sup_{0\leq \lambda \leq 1}\left[ \sup_{(1-t)^{\frac{1}{p}}(1-\lambda)^{\frac{p-1}{p}}x + t^{\frac{1}{p}}\lambda^{\frac{p-1}{p}}y}  f(x)^{(1-t)^{\frac{1}{p}}(1-\lambda)^{\frac{p-1}{p}}}g(y)^{t^{\frac{1}{p}}\lambda^{\frac{p-1}{p}}} \right]\\
 &\geq e^{-n} \sup_{0 \leq \lambda \leq 1}\left(\|f\|_{\infty}^{(1-t)^{\frac{1}{p}}(1-\lambda)^{\frac{p-1}{p}}} \cdot \|g\|_{\infty}^{t^{\frac{1}{p}}\lambda^{\frac{p-1}{p}}}\right) = e^{-n}\|h\|_{\infty}.
\end{align*}
Hence, the inclusion \eqref{e:awesomeinclusion} holds, and the proof is complete. 

\end{proof}
In particular, we have the special inequality for measures and convex bodies instead of measures for functions as follows; that is, the extension of the $L_p$-BMI for product measures with quasi-concave densities to a log concave measure with a constant bound $\frac{1}{C^p}.$
\begin{corollary}\label{t:LpGZmeasures}
	Let $p \geq 1$, $t \in [0,1]$, and $\mu$ be a $\log$-concave measure on $\R^n$.  Then there exists a universal constant $C>1$ such that, for any $K,L \in \mathcal{K}_{(o)}^n$, one has 
	\begin{equation}\label{e:LpGZmeasures}
	\mu((1-t) \cdot_p K +_p t \cdot_p L)^{\frac{p}{n}} \geq \frac{1}{C^p}\left[(1-t) \mu(K)^{\frac{p}{n}} + t \mu(L)^{\frac{p}{n}} \right]. 
	\end{equation}
\end{corollary}
\begin{proof} In fact, let $\phi$ denote the density of the measure $\mu$ consider the functions $f = \phi \cdot 1_K$ and $g = \phi \cdot 1_L$.  Then
	\[
	h = (1-t) \cdot_{p,\infty} f \oplus_{p,\infty} t \cdot_{p,\infty} g = \phi \cdot 1_{(1-t) \cdot_p K +_p t \cdot_p L}.
	\]
	Then by applying inequality \eqref{e:lpfuncxtionalGZ} to the triple of functions $f,g,h$, we obtain inequality \eqref{e:LpGZmeasures}, as desired. 
\end{proof}

\begin{remark} \label{t:GZremark}
	The establishment of the above inequality (\ref{e:LpGZmeasures}) up to some universal constant turns out to be simple for measure $\mu$ with radial decay for the Minkowski convex combination, i.e.,  $\mu(t K) \geq t^n \mu(K)$ for convex sets $K$ containing the origin and $t \in [0,1]$.  Example of measures are those with $\left(\frac{1}{s}\right)$-concave densities that assume their maximum at the origin.  Therefore, for any convex bodies $K,L \in \mathcal{K}^n_{(o)}$, $p \geq 1$, and $t \in [0,1]$, and measure $\mu$ on $\R^n$ that had radial decay, we see that 
	\begin{align*}
	\mu((1-t) \cdot_p K +_p t \cdot_p L)^{\frac{p}{n}} &\geq \max\left\{\mu((1-t) \cdot_p K)^{\frac{p}{n}}, \mu(t \cdot_p L)^{\frac{p}{n}} \right\}\\
	&\geq \frac{1}{2}\left(\mu((1-t) \cdot_p K)^{\frac{p}{n}} +\mu(t \cdot_p L)^{\frac{p}{n}} \right)\\
	&\geq \frac{1}{2}\left((1-t)\mu(K)^{\frac{p}{n}} +t \mu(L)^{\frac{p}{n}} \right),
	\end{align*}
where we have used the fact that $(1-t) \cdot_p K, t \cdot_p L \subset (1-t) \cdot_p K +_p t \cdot_p L$ for $p\geq 1$. 
\end{remark}
\section*{Acknowledgements} The authors would like to thank Professors Artem Zvavitch, Andrea Colesanti, and Jes\'us Yepes Nicholas for their reading of the manuscript and for providing valuable suggestions and discussion on the content of the paper.  The authors would like to thank BIRS in Banff, Canada and the Mathematics Department at the University of Alberta, where part of this manuscript was written, for their warm hospitality and inviting environment. Finally, the authors would like to thank Professor Galyna Livshyts for holding the workshop discussion session in Georgia Institute of Technology where the project first ignited.

\end{document}